\documentclass[11pt,letterpaper]{article}
\usepackage[lmargin=1.0in,rmargin=1.0in,bottom=1.0in,top=1.0in,twoside=False]{geometry}

\usepackage[utf8]{inputenc}
\usepackage[T1]{fontenc}
\usepackage{lmodern}

\usepackage{fullpage,amssymb,amsmath}
\usepackage{graphicx}
\usepackage{tikz}
\usetikzlibrary{shapes}

\usepackage{xcolor}
\usepackage{mathtools}
\usepackage{microtype}
\usepackage{amsfonts}
\usepackage{comment}
\usepackage[english]{babel}
\usepackage{mathrsfs}
\usepackage[ruled,vlined]{algorithm2e}
\usepackage{blindtext}
\usepackage{thmtools}
\usepackage{thm-restate}
\usepackage{booktabs}
\usepackage[sort&compress,numbers]{natbib}

\usepackage{trimspaces}
\usepackage{nccfoots}
\usepackage{setspace}
\usepackage{inconsolata}
\usepackage{libertine}
\usepackage[absolute]{textpos}

\usepackage[backgroundcolor = blue!50]{todonotes}
\usepackage{longtable}

\definecolor{blue}{rgb}{0.1,0.2,0.5}
\definecolor{brown}{rgb}{0.6,0.6,0.2}
\usepackage[ocgcolorlinks, linkcolor={blue}, citecolor={brown}]{hyperref}
\usepackage[amsmath,amsthm,thmmarks,hyperref]{ntheorem}
\usepackage{enumerate}
\usepackage{latexsym}
\usepackage{changepage}
\usepackage{boxedminipage}

\usepackage[nameinlink]{cleveref}

\crefformat{page}{#2page~#1#3}%
\Crefformat{page}{#2Page~#1#3}%
\crefformat{equation}{#2(#1)#3}%
\Crefformat{equation}{#2(#1)#3}%
\crefformat{figure}{#2Figure~#1#3}%
\Crefformat{figure}{#2Figure~#1#3}%
\crefformat{section}{#2Section~#1#3}
\Crefformat{section}{#2Section~#1#3}
\crefformat{chapter}{#2Chapter~#1#3}
\Crefformat{chapter}{#2Chapter~#1#3}
\crefformat{chapter*}{#2Chapter~#1#3}
\Crefformat{chapter*}{#2Chapter~#1#3}
\crefformat{part}{#2Part~#1#3}
\Crefformat{part}{#2Part~#1#3}
\crefformat{enumi}{#2(#1)#3}
\Crefformat{enumi}{#2(#1)#3}

\usepackage[mathlines]{lineno}

\newcommand*\patchAmsMathEnvironmentForLineno[1]{%
  \expandafter\let\csname old#1\expandafter\endcsname\csname #1\endcsname
  \expandafter\let\csname oldend#1\expandafter\endcsname\csname end#1\endcsname
  \renewenvironment{#1}%
     {\linenomath\csname old#1\endcsname}%
     {\csname oldend#1\endcsname\endlinenomath}}%
\newcommand*\patchBothAmsMathEnvironmentsForLineno[1]{%
  \patchAmsMathEnvironmentForLineno{#1}%
  \patchAmsMathEnvironmentForLineno{#1*}}%
\AtBeginDocument{%
  \patchBothAmsMathEnvironmentsForLineno{equation}%
  \patchBothAmsMathEnvironmentsForLineno{align}%
  \patchBothAmsMathEnvironmentsForLineno{flalign}%
  \patchBothAmsMathEnvironmentsForLineno{alignat}%
  \patchBothAmsMathEnvironmentsForLineno{gather}%
  \patchAmsMathEnvironmentForLineno{split}
  \patchBothAmsMathEnvironmentsForLineno{multline}}

\theoremnumbering{arabic}
\theoremstyle{plain}
\theoremsymbol{}
\theorembodyfont{\itshape}
\theoremheaderfont{\normalfont\bfseries}
\theoremseparator{.}

\newtheorem{theorem}{Theorem}
\crefformat{theorem}{#2Theorem~#1#3}
\Crefformat{theorem}{#2Theorem~#1#3}
\crefformat{cthmin}{#2Theorem~#1#3}
\Crefformat{cthmin}{#2Theorem~#1#3}

\crefformat{definition}{#2Definition~#1#3}
\Crefformat{definition}{#2Definition~#1#3}

\newcommand{\newtheoremwithcrefformat}[2]{%
  \newtheorem{#1}[theorem]{#2}%
  \crefformat{#1}{##2\MakeUppercase#1~##1##3}%
  \Crefformat{#1}{##2\MakeUppercase#1~##1##3}%
}
\newcommand{\newseptheoremwithcrefformat}[2]{%
  \newtheorem{#1}{#2}%
  \crefformat{#1}{##2\MakeUppercase#1~##1##3}%
  \Crefformat{#1}{##2\MakeUppercase#1~##1##3}%
}
\newcommand{\newclaimwithcrefformat}[2]{%
  \newtheorem{#1}{#2}[theorem]%
  \crefformat{#1}{##2\MakeUppercase#1~##1##3}%
  \Crefformat{#1}{##2\MakeUppercase#1~##1##3}%
}

\newtheoremwithcrefformat{lemma}{Lemma}

\newtheoremwithcrefformat{proposition}{Proposition}
\newtheoremwithcrefformat{observation}{Observation}
\newseptheoremwithcrefformat{conjecture}{Conjecture}
\newseptheoremwithcrefformat{question}{Question}
\newtheoremwithcrefformat{corollary}{Corollary}
\newclaimwithcrefformat{claim}{Claim}
\newclaimwithcrefformat{redrule}{Reduction Rule}

\theoremstyle{definition}

\newtheorem*{remark*}{Remark}
\theorembodyfont{\upshape}

\theoremstyle{nonumberplain}
\theoremheaderfont{\scshape}
\theorembodyfont{\normalfont}
\theoremsymbol{\ensuremath{\square}}

\theoremsymbol{\ensuremath{\lrcorner}}

\newcommand{\problemdef}[3]{
	\begin{center}
		\begin{boxedminipage}{.99\textwidth}
			\textsc{{#1}}\\[2pt]
			\begin{tabular}{ r p{0.8\textwidth}}
				\textit{Instance:} & {#2}\\
				\textit{Question:} & {#3}
			\end{tabular}
		\end{boxedminipage}
	\end{center}
}

\newcommand{\NP}{\textsf{NP}}

\newcommand{\cA}{\mathcal{A}}

\newcommand{\cC}{\mathcal{C}}

\newcommand{\cE}{\mathcal{E}}
\newcommand{\cF}{\mathcal{F}}
\newcommand{\cG}{\mathcal{G}}
\newcommand{\cH}{\mathcal{H}}
\newcommand{\cI}{\mathcal{I}}

\newcommand{\cM}{\mathcal{M}}

\newcommand{\cS}{\mathcal{S}}

\newcommand{\cX}{\mathcal{X}}
\newcommand{\cY}{\mathcal{Y}}

\newcommand{\tG}{\widetilde{G}}

\newcommand{\tX}{\widetilde{X}}
\newcommand{\tE}{\widetilde{E}}

\newcommand{\N}{\mathbb{N}}

\newcommand{\vphi}{\varphi}

\renewcommand{\epsilon}{\varepsilon}
\renewcommand{\phi}{\varphi}
\newcommand{\Oh}{\mathcal{O}}
\newcommand{\wei}{\mathfrak{w}}

\renewcommand{\leq}{\leqslant}
\renewcommand{\geq}{\geqslant}

\renewcommand{\setminus}{-}

\newcommand{\MIS}{\textsc{MIS}\xspace}
\newcommand{\MWM}{\textsc{MWM$^*$}\xspace}
\newcommand{\coloring}[1]{\ensuremath{#1}\textsc{-Coloring}\xspace}
\newcommand{\colorext}[1]{\ensuremath{#1}\textsc{-ColoringExt}\xspace}
\newcommand{\lcoloring}[1]{\textsc{List} \ensuremath{#1}\textsc{-Coloring}\xspace}

\newcommand{\tw}[1]{{\operatorname{tw}(#1)}}

\newenvironment{claimproof}[1][Proof of Claim.]{\noindent {\emph{#1} }}{\hfill$\blacksquare$\medskip}

\declaretheorem[sibling=theorem]{lemma}

\newcommand{\itte}{\textsc{ITTE}\xspace}

\title{Computing homomorphisms in hereditary graph classes:\\
the peculiar case of the 5-wheel and graphs with no long claws%
\footnote{MD, MP, and PRz were supported by Polish National Science Centre grant no. 2018/31/D/ST6/00062.
KO was supported by Polish National Science Centre grant no. 2021/41/N/ST6/01507}
}
\author{Michał Dębski \and Zbigniew Lonc \and Karolina Okrasa \and Marta Piecyk \and Paweł Rzążewski}
\date{}
\begin{document}
\maketitle

\begin{abstract}
For graphs $G$ and $H$, an $H$-coloring of $G$ is an edge-preserving mapping from $V(G)$ to $V(H)$.
In the \textsc{$H$-Coloring} problem the graph $H$ is fixed and we ask whether an instance graph $G$ admits an $H$-coloring.
A generalization of this problem is \textsc{$H$-ColoringExt}, where some vertices of $G$ are already mapped to vertices of $H$ and we ask if this partial mapping can be extended to an $H$-coloring.

We study the complexity of variants of \textsc{$H$-Coloring} in $F$-free graphs, i.e., graphs excluding a fixed graph $F$ as an induced subgraph. For integers $a,b,c \geq 1$, by $S_{a,b,c}$ we denote the graph obtained by identifying one endvertex of three paths on $a+1$, $b+1$, and $c+1$ vertices, respectively. For odd $k \geq 5$, by $W_k$ we denote the graph obtained from the $k$-cycle by adding a universal vertex.

As our main algorithmic result we show that \textsc{$W_5$-ColoringExt} is polynomial-time solvable in $S_{2,1,1}$-free graphs. 
This result exhibits an interesting non-monotonicity of \textsc{$H$-ColoringExt} with respect to taking induced subgraphs of $H$. Indeed, $W_5$ contains a triangle, and \textsc{$K_3$-Coloring}, i.e., classical 3-coloring,
is \textsf{NP}-hard already in claw-free (i.e., $S_{1,1,1}$-free) graphs.
Our algorithm is based on two main observations:
\begin{enumerate}
\item \textsc{$W_5$-ColoringExt} in $S_{2,1,1}$-free graphs can be in polynomial time reduced to a variant of the problem of finding an independent set intersecting all triangles, and
\item the latter problem can be solved in polynomial time in $S_{2,1,1}$-free graphs.
\end{enumerate}

We complement this algorithmic result with several negative ones. In particular, we show that \textsc{$W_5$-ColoringExt} 
is \textsf{NP}-hard in $S_{3,3,3}$-free graphs. This is again uncommon, as usually problems that are \textsf{NP}-hard in $S_{a,b,c}$-free graphs for some constant $a,b,c$ are already hard in claw-free graphs.

\end{abstract}

\newpage
\section{Introduction}
Many computationally hard problems become tractable when restricted to instances with some special properties.
In case of graph problems, typical families of such special instances come from forbidding certain substructures.
For example, for a family $\cF$ of graphs, a graph $G$ is \emph{$\cF$-free} if it does not contain any graph from $\cF$ as an induced subgraph. If $\cF$ contains a single graph $F$, then we write $F$-free instead of $\{F\}$-free.
Classes defined by forbidden induced subgraphs are \emph{hereditary}, i.e., closed under vertex deletion.
Conversely, every hereditary class of graphs can be uniquely characterized by a (possibly infinite) set of minimal forbidden induced subgraphs.

Let us define two families of graphs that play a special role as forbidden induced subgraphs.
For $t\geq 1$, by $P_t$ we denote the path in $t$ vertices.
For $a,b,c\geq 1$, by $S_{a,b,c}$ we mean the graph consisting of three paths $P_{a+1},P_{b+1},P_{c+1}$ with one endvertex identified. Each graph $S_{a,b,c}$ is called a \emph{subdivided claw}.
The smallest subdivided claw, i.e., $S_{1,1,1}$ is the \emph{claw}, and the graph $S_{2,1,1}$ is sometimes called the \emph{fork} or the \emph{chair}.
Finally, let $\cS$ denote the class of graphs whose every component is a path or a subdivided claw. 

\paragraph{\MIS and \coloring{k} in $F$-free graphs.}
Two problems, whose complexity in hereditary graph classes attracts significant attention, are \textsc{Max (Weighted) Independent Set} (denoted by \MIS) and \coloring{k}. Let us briefly survey known results, focusing on $F$-free graphs for connected $F$.
 
By the observation of Alekseev~\cite{Alekseev82}, \MIS is \NP-hard in $F$-free graphs, unless $F \in \cS$.
On the positive side, polynomial-time algorithms are known for some small graphs $F \in \cS$.
If $F = P_t$, then the polynomial-time algorithm for $t \leq 5$ was provided by Lokshtanov et al.~\cite{LVV},
which was later extended to $t \leq 6$ by Grzesik et al.~\cite{GrzesikKPP19}.
The case of $t=7$ remains open and the general belief is that for every $t$ the problem is polynomial-time solvable.
The evidence is given by the existence of quasipolynomial-time algorithms~\cite{GartlandL20,DBLP:conf/sosa/PilipczukPR21}.

The polynomial-time algorithm for \MIS in claw-free graphs~\cite{Sbihi80,Minty80} can be obtained by an extension of the augmenting path approach used for finding largest matchings~\cite{edmonds_1965}; note that a maximum matching is precisely a largest independent set in the line graph, and line graphs are in particular claw-free.
There are also more modern approaches, based on certain decompositions of claw-free graphs~\cite{FaenzaOS14,DBLP:journals/mp/NobiliS21}.
A polynomial-time algorithm for \MIS in $S_{2,1,1}$-free graphs was first obtained by Alekseev~\cite{Al04} (only for the unweighted case),
and later an arguably simpler algorithm was provided by Lozin and Milani\v{c}~\cite{DBLP:journals/jda/LozinM08} (also for the weighted case).
Again, the complexity of \MIS in $F$-free graphs for larger subdivided claws $F$ remains open, but all these cases are believed to be tractable.
This belief is supported by the existence of a subexponential-time algorithm~\cite{DBLP:journals/corr/abs-1907-04585,DBLP:journals/corr/abs-2203-04836}, a QPTAS~\cite{chudnovsky2020quasi,DBLP:journals/corr/abs-1907-04585,DBLP:journals/corr/abs-2203-04836}, or a polynomial-time algorithm in the bounded-degree case~\cite{ACDR21}.

If it comes to \coloring{k}, then it follows from known results that if $F$ is not a forest of paths,
then for every $k \geq 3$ the problem is \NP-hard in $F$-free graphs~\cite{DBLP:journals/corr/Golovach0PS14,DBLP:journals/jal/LevenG83,DBLP:journals/cpc/Emden-WeinertHK98,Holyer1981TheNO}.
The complexity of \coloring{k} in $P_t$-free graphs is quite well understood.
For $t =5$, the problem is polynomial-time solvable for every constant $k$~\cite{DBLP:journals/algorithmica/HoangKLSS10}.
If $k \geq 5$, then the problem is \NP-hard already in $P_6$-free graphs~\cite{DBLP:journals/ejc/Huang16}.
The case of $k=4$ is polynomial-time solvable for $t \leq 6$~\cite{DBLP:conf/soda/SpirklCZ19} and \NP-hard for $t \geq 7$~\cite{DBLP:journals/ejc/Huang16}.
The case of $k=3$ is much more elusive.
We know a polynomial-time algorithm for $P_7$-free graphs~\cite{DBLP:journals/combinatorica/BonomoCMSSZ18} and the cases for all $t \geq 8$ are open.
The general belief that they should be tractable is again supported by the existence of a quasipolynomial-time algorithm~\cite{DBLP:conf/sosa/PilipczukPR21}.

Let us mention two generalizations of \coloring{k}.
In the \colorext{k} problem we are given a graph $G$ with a subset of its vertices colored using $k$ colors,
and we ask whether this partial coloring can be extended to a proper $k$-coloring of $G$.
In the even more general \lcoloring{k} problem, each vertex of the instance graph $G$ is equipped with a subset of $\{1,\ldots,k\}$ called \emph{list}, and we ask whether there exists a proper $k$-coloring of $G$ respecting all lists.
Clearly any tractability result for a more general problem implies the same result for a less general one,
and any hardness result for a less general problem implies the same hardness for more general variants.
In almost all mentioned cases the algorithms for \coloring{k} generalize to \lcoloring{k}.
The only exception is the case $k=4$ and $t=6$: the polynomial time algorithm for \coloring{4} in $P_6$-free graphs can be generalized to \colorext{4}~\cite{DBLP:conf/soda/SpirklCZ19}, but \lcoloring{4} in this class is \NP-hard~\cite{DBLP:journals/iandc/GolovachPS14}

\paragraph{Minimal obstructions.}
One of the ways of designing polynomial-time algorithms for \coloring{k} is to check is the instance graph contains some (hopefully small) subgraph that is not $k$-colorable. This approach is formalized by the notion of \emph{critical graphs}.
A graph $G$ is \emph{$(k+1)$-vertex critical} if it is not $k$-colorable, but its every induced subgraph is. Such graphs can be thought of \emph{minimal obstructions} to $k$ coloring: a graph $G$ is $k$-colorable if and only if it does not contain any $(k+1)$-vertex-critical graph as an induced subgraph. 
Thus if for some hereditary class $\cG$ of graphs, the number of $(k+1)$-vertex critical graphs is finite,
we immediately obtain a polynomial-time algorithm for \coloring{k} in graphs from $\cG$.
Indeed, it is sufficient
to check if the instance graph contains any $(k+1)$-vertex-critical induced subgraph, which can be done in polynomial time by brute-force.
Such an algorithm, in addition to solving the instance, provides a \emph{certificate} in case of the negative answer -- a constant-size subgraph which does not admit a proper $k$-coloring.
Thus the question whether for some class $\cG$, the number of $(k+1)$-vertex critical graphs is finite,
can be seen as a refined analysis of the polynomial cases of \coloring{k}. 

The finiteness of the families of $(k+1)$-vertex critical in $F$-free graphs is fully understood.
Recall that the only interesting (i.e., not known to be \NP-hard) cases are for $F$ being a forest of paths.
Again focusing on connected $F$, i.e., \coloring{k} of $P_t$-free graphs,
we know that the families of minimal obstructions are finite for $t \leq 6$ and $k=4$~\cite{DBLP:journals/jct/ChudnovskyGSZ20},
and for $t \leq 4$ and any $k$. The latter result follows from the fact that $P_4$-free graphs are perfect and thus the only obstruction for $k$-coloring is $K_{k+1}$.
In all other cases there are constructions of infinite families of minimal obstructions~\cite{DBLP:journals/jct/ChudnovskyGSZ20,DBLP:journals/dam/HoangMRSV15}.

\paragraph{Graph homomorphisms in $F$-free graphs.}
A homomorphism from a graph $G$ to a graph $H$ is a mapping from $V(G)$ to $V(H)$ that preserves edges, i.e.,
the image of every edge of $G$ is an edge of $H$.
Note that if $H$ is the complete graph on $k$ vertices, then homomorphisms to $H=K_k$ are exactly proper $k$-colorings.
For this reason homomorphisms to $H$ are called \emph{$H$-colorings}, we will also refer to vertices of $H$ as \emph{colors}.
In the \coloring{H} problem the graph $H$ is fixed and we need to decide whether an instance graph $G$ admits a homomorphism to $H$.
By the analogy to coloring, we also define more general variants, i.e., \colorext{H} and \lcoloring{H}.
In the former problem we ask whether a given partial mapping from vertices of an instance graph $G$ to $V(H)$ can be extended to a homomorphism,
and in the latter one each vertex of $G$ is equipped with a list which is a subset of $V(H)$ and we look for a homomorphism respecting all lists.

The complexity of \coloring{H} was settled by Hell and Ne\v{s}et\v{r}il~\cite{DBLP:journals/jct/HellN90}: the problem is polynomial-time solvable if $H$ is bipartite or has a vertex with a loop, and \NP-hard otherwise. The dichotomy is also known for \lcoloring{H}~\cite{FEDER1998236,DBLP:journals/combinatorica/FederHH99,DBLP:journals/jgt/FederHH03}: this time the tractable cases are the so-called \emph{bi-arc graphs}.
The case of \colorext{H} is more tricky. The classification follows from the celebrated proof of the CSP complexity dichotomy~\cite{DBLP:conf/focs/Bulatov17,DBLP:journals/jacm/Zhuk20},
but the graph-theoretic description of polynomial cases is unknown (and probably difficult to obtain).

We are very far from understanding the complexity of variants of \coloring{H} in $F$-free graphs.
Chudnovsky et al.~\cite{DBLP:conf/esa/ChudnovskyHRSZ19} showed that \lcoloring{C_k} for $k \in \{5,7\} \cup [9,\infty)$ is polynomial-time solvable in $P_9$-free graphs. On the negative side, they showed that for every $k \geq 5$ the problem is \NP-hard in $F$-free graphs,
unless $F \in \cS$.
This negative result was later extended by Piecyk and Rzążewski~\cite{DBLP:conf/stacs/PiecykR21} who showed that if $H$ is not a bi-arc graph,  then \lcoloring{H} is \NP-hard and cannot be solved in subexponential time (assuming the ETH) in $F$-free graphs, unless $F \in \cS$.

The case of forbidden path or subdivided claw was later investigated by Okrasa and Rzążewski~\cite{DBLP:conf/stacs/OkrasaR21}.
They defined a class of \emph{predacious graphs}
and showed that if $H$ is not predacious, then for every $t$, the \lcoloring{H} problem
can be solved in quasipolynomial time in $P_t$-free graphs.
On the other hand, for every predacious $H$ there exists $t$
for which \lcoloring{H} cannot be solved in subexponential time in $P_t$-free graphs unless the ETH fails.
They also provided some partial results for the case of forbidden subdivided claws.

Finally, Chudnovsky et al.~\cite{DBLP:conf/esa/ChudnovskyKPRS20} considered a generalization of \lcoloring{H} in $P_5$-free graphs and related classes.

The notion of vertex-critical graphs can be naturally translated to $H$-colorings.
A graph $G$ is a \emph{minimal $H$-obstruction} if it is not $H$-colorable, but its every induced subgraph is.
Minimal $H$-obstructions in restricted graph classes were studied by some authors, but the results are rather scattered~\cite{DBLP:journals/dam/KaminskiP19,DBLP:conf/caldam/BeaudouFN20}.

\paragraph{Our motivation.}
Let us point a substantial difference between working with \coloring{H}and working with \lcoloring{H} (with \colorext{H} being somewhere between, but closer to \coloring{H}).
The \lcoloring{H} problem enjoys certain monotonicity: if $H'$ is an induced subgraph of $H$,
then every instance of \lcoloring{H'} can be seen as an instance of \lcoloring{H}, where no vertex from $V(H) \setminus V(H')$ appears in any list. Thus any tractability result for \lcoloring{H} implies the analogous result for \lcoloring{H'},
while any hardness result for \lcoloring{H'} applies also to \lcoloring{H}.
In particular, all hardness proofs for \lcoloring{H} follow the same pattern: first we identify a (possibly infinite) family $\cH$ of ``minimal hard cases''
and then show hardness of \lcoloring{H'} for every $H' \in \cH$. This implies that \lcoloring{H} is hard unless $H$ is $\cH$-free.
The lists are also useful in the design of algorithms: for example if for some reason we decide that some vertex $v \in V(G)$ must be mapped to $x \in V(H)$, we can remove from the lists of neighbors of $v$ all non-neighbors of $x$, and then delete $v$ from the instance graph.
This combines well with e.g. branching algorithms or divide-\&-conquer algorithms based on the existence of separators.

In contrast, when coping with \coloring{H} we need to think about the global structure of $H$.
This makes working with this variant of the problem much more complicated. In particular, hardness proofs often employ certain algebraic tools~\cite{DBLP:journals/jct/HellN90,DBLP:journals/tcs/Bulatov05}, which in turn do not combine well with the world of $F$-free graphs.
However, note that the complexity dichotomy for \coloring{H} \emph{is still monotone} with respect to taking induced subgraphs of $H$: 
the minimal \NP-hard cases are odd cycles.

An interesting example of non-monotonicity of \coloring{H} was provided by Feder and Hell in an unpublished manuscript~\cite{edgelists}.
For an odd integer $k \geq 5$, let $W_k$ denote the \emph{$k$-wheel}, i.e., the graph obtained from the $k$-cycle by adding a universal vertex.
Feder and Hell proved that \coloring{W_5} is polynomial-time solvable in line graphs.
This is quite surprising as $W_5$ contains a triangle and \coloring{K_3}, i.e., \coloring{3}, is \NP-hard in line graphs~\cite{Holyer1981TheNO}.
(Note that this implies that \lcoloring{W_5} is \NP-hard in line graphs.)
Feder and Hell also proved that for any odd $k \geq 7$, the \coloring{W_k} problem is \NP-hard in line graphs~\cite{edgelists}.

\paragraph{Our contribution.} In this paper we study to which extent the result of Feder and Hell~\cite{edgelists} can be generalized.
We provide an algorithm and a number of lower bounds, each of a different kind.

The main algorithmic contribution of our paper is the following theorem.
\begin{restatable}{theorem}{mainthm}
\label{thm:mainthm}
The \colorext{W_5} problem can be solved in polynomial time in $S_{2,1,1}$-free graphs.
\end{restatable}

Let us sketch the outline of the proof.
Perhaps surprisingly, despite the fact that graph homomorphisms generalize graph colorings,
our approach is much closer to the algorithms for \MIS.

Consider any homomorphism $\phi$ from $G$ to $W_5$ and let $X$ be the set of vertices of $G$ mapped by $\phi$ to the universal vertex of $W_5$.
We notice that $X$ is independent, and $G-X$ admits a homomorphism to $C_5$.
This is exactly how we look at the problem: we aim to find an independent set $X$ whose removal makes the graph $C_5$-colorable
(and make sure that the precoloring of vertices is respected).

So let us focus on describing the structure of $G-X$, i.e., recognizing $C_5$-colorable graphs.
Note that here we need to use the fact that our graph is $S_{2,1,1}$-free,
as \coloring{C_5} is \NP-hard in general graphs.
As a warm-up let us assume that our instance is claw-free and forget about precolored vertices.
We observe that every $C_5$-colorable graph must be triangle-free, as there is no homomorphism from $K_3$ to $C_5$.
But since $G-X$ is $\{S_{1,1,1},K_3\}$-free, it must be of maximum degree at most 2, i.e.,
every component of $G-X$ is a path or a cycle with at least 4 vertices.
It is straightforward to verify that such graphs always admit a homomorphism to $C_5$.
Therefore in claw-free graphs, solving \coloring{W_5} boils down to finding an independent set intersecting all triangles.

We extend this simple observation in two ways. First, we show that the same phenomenon occurs in $S_{2,1,1}$-free graphs:
the only $S_{2,1,1}$-free minimal $C_5$-obstruction is the triangle.
Second, we show that the same way we can handle precolored vertices:
the only no-instances of \colorext{C_5} in $\{S_{2,1,1},K_3\}$-free graphs can be easily recognized.
Based on this result we show that \colorext{W_5} in $S_{2,1,1}$-free graphs can be in polynomial time
reduced to the \textsc{Independent Triangle Transversal Extension} (\itte) problem on an induced subgraph of the original instance.
Here, the ``extension'' means that some vertices of our instance can be prescribed to be in $X$ or outside $X$.
Thus from now on we focus on solving \itte in $S_{2,1,1}$-free graphs.
Note that we still need to use the fact that our instances are $S_{2,1,1}$-free, as \itte is \NP-hard in general graphs~\cite{DBLP:journals/combinatorics/Farrugia04}.

We start with the case that our instance graph $G$ is claw-free. 
We use the result of Chudnovsky and Seymour~\cite{DBLP:journals/jct/ChudnovskyS08e} who show that each claw-free graph admits certain decomposition called a \emph{strip structure}.
Roughly speaking, this means that $G$ ``resembles'' the line graph of some graph $D$: the vertices of $G$ can be partitioned
into sets $\eta(e)$ assigned to the edges $e$ of $D$, such that (i) for each $e \in E(D)$ the set $\eta(e)$ induces a subgraph of $G$ with a simple structure
and (ii) the interactions between sets $\eta(e)$ and $\eta(f)$ for distinct edges $e,f$ are well-defined.
Due to property (i), for each edge $e$ of $D$ we can solve the problem in the subgraph of $G$ induced by $\eta(e)$ in polynomial time.
Then we can use property (ii) to combine these partial solutions into the final one by finding an appropriate matching in an auxiliary graph derived from $D$.

As the final step, we lift our algorithm for claw-free graphs to the class of $S_{2,1,1}$-free graphs.
We use the observation of Lozin and Milani\v{c}~\cite{DBLP:journals/jda/LozinM08}:
they show that if $G$ is $S_{2,1,1}$-free and \emph{prime}, then every \emph{prime} induced subgraph of the graph obtained from $G$ by removing any vertex and its neighbors is claw-free (the exact definition of prime graphs can be found in \cref{sec:tools}).
Using this observation we can reduce the \itte problem on $S_{2,1,1}$-free graphs to the same problem in the class of claw-free graphs,
which we already know how to solve in polynomial time.

Combining the reduction from \colorext{W_5} in $S_{2,1,1}$-free graphs to \itte in $S_{2,1,1,}$-free graphs
and the polynomial-time algorithm for \itte in $S_{2,1,1}$-free graphs we obtain \cref{thm:mainthm}.
Let us point out that the frontier of the complexity of \colorext{W_5} in $S_{a,b,c}$-free graphs is the same as for \MIS: the minimal open cases are $S_{3,1,1}$ and $S_{2,2,1}$.

\medskip
In the remainder of the paper we investigate several possible generalizations of \cref{thm:mainthm} and show a number of negative results.

First, one can ask whether a simpler algorithm, at least for \coloring{W_5}, can be obtained by showing that the family of minimal $S_{2,1,1}$-free $W_5$-obstructions is finite. We show that this is not the case: we construct an infinite family of minimal $W_5$-obstructions (actually, $W_k$-obstructions for every odd $k \geq 5$) which are even claw-free.

Next, we show that \colorext{W_5} in $S_{3,3,3}$-free graphs of maximum degree 5 is \NP-hard and cannot be solved in subexponential time, unless the ETH fails.
This should be contrasted with the fact that for any $a,b,c$, the \MIS problem in $S_{a,b,c}$-free graphs can be solved in subexponential time~\cite{DBLP:journals/corr/abs-1907-04585,DBLP:journals/corr/abs-2203-04836},
and even in polynomial time, if the instance is of bounded maximum degree~\cite{ACDR21}.
These facts are an evidence that \colorext{W_5} in $S_{a,b,c}$-free graphs for large $a,b,c$ cannot be solved with the tools developed for \MIS, as it is the case for $a=2$ and $b=c=1$.
Again, we find our hardness result quite surprising, as typically problems that are hard in $S_{a,b,c}$-free graphs for some fixed $a,b,c$ are already had in claw-free graphs.

Finally, we consider the complexity of variants of \coloring{W_k} in $F$-free graphs.
We remark that for every $k$, the \lcoloring{W_k} problem can be solved in polynomial time in $P_5$-free graphs~\cite{DBLP:conf/esa/ChudnovskyKPRS20}
Furthermore for every $k$ the graph $W_k$ is not predacious, so by the result of Okrasa and Rzążewski~\cite{DBLP:conf/stacs/OkrasaR21}, 
\lcoloring{W_k} can be solved in quasipolynomial time in $P_t$-free graphs, for every fixed $t$.
We show that forbidding any other connected graph leads to an \NP-hard problem,
with possible exception of $\coloring{W_5}$ in $S_{a,b,c}$-free graphs.

\begin{restatable}{theorem}{thmhardness}
\label{thm:hardness}
Let $F$ be a connected graph.
\begin{enumerate}
\item If $F$ is not a path nor a subdivided claw, then the \coloring{W_5} problem is \NP-hard in $F$-free graphs and cannot be solved in subexponential time, unless the ETH fails.
\item Let $k \geq 7$ be odd. If $F$ is not a path, then the \coloring{W_k} problem is \NP-hard in $F$-free graphs and cannot be solved in subexponential time, unless the ETH fails.
\end{enumerate}
\end{restatable}

The paper is concluded with pointing our some open problems and possible directions for future research.


\begin{remark*}
The curious reader might wonder why we only consider $k$-wheels for odd $k\geq 5$.
The 3-wheel is exactly $K_4$, and homomorphisms to $K_4$ are exactly proper 4-colorings.
As mentioned above, both \coloring{4} and \colorext{4} are well studied in hereditary graph classes and behave substantially differently than \coloring{W_k} for odd $k \geq 5$.
On the other hand, if $k$ is even, then the \coloring{W_k} problem is equivalent to the \coloring{3} problem:
a graph admits a homomorphism to $W_k$ if and only if it is 3-colorable (this follows from the fact that $K_3$ is the \emph{core} of $W_k$~\cite[Section 1.4]{hell2004graphs}).
Thus it only makes sense to consider variants of \coloring{W_k} in $F$-free graphs, where $F$ is a path or a forest of paths. However, as $W_k$ is non-predacious, it follows from the result of Okrasa and Rzążewski~\cite{DBLP:conf/stacs/OkrasaR21} that then even \lcoloring{W_k} is quasipolynomial-time solvable in $F$-free graphs.
\end{remark*}

\section{Preliminaries}\label{sec:prelim}
For an integer $k$, by $[k]$ we denote the set $\{1,2,\ldots,k\}$. 

Let $G$ be a graph, $v$ be a vertex and $X$ be a set of vertices.
By $N_G(v)$ we denote the set of neighbors of $v$, and by $N_G[v]$ we denote the set $N_G(v) \cup \{v\}$.
If $G$ is clear from the context, we omit the subscript, and write, respectively, $N(v)$, and $N[v]$. 
By $G[X]$ we denote the subgraph induced by $X$, and by $G - X$ we denote $G[V(G) \setminus X]$.
By $\overline{G}$ we denote the complement of $G$.

We write $\phi : G \to H$ to indicate that $\phi$ is a homomorphism from $G$ to $H$.
We also write $G \to H$ to indicate that some homomorphism from $G$ to $H$ exists.

For a fixed graph $H$, an instance of \colorext{H} is a triple $(G,U,\vphi)$,
where $G$ is a  graph, $U$ is a subset of vertices of $G$, and $\vphi$ is a function that maps vertices from $U$ to elements of $V(H)$. We ask whether there exists a homomorphism $\psi : G \to H$ such that $\psi|_{U} = \phi$.


For a $k$-wheel $W_k$, we will always denote the consecutive vertices of the induced $k$-cycle in $W_k$ by $1,2,3,\ldots,k$, and the universal vertex by $0$.
The following observation is straightforward and will be used implicitly throughout the paper.

\begin{observation}
Let $k \geq 5$ be an odd integer.
Let $\phi$ be a homomorphism from a graph $G$ to $W_k$.  Let $X$ be the set of vertices of $G$ mapped by $\phi$ to 0. Then the following properties are met:
\begin{enumerate}
\item $G$ is 4-colorable,
\item $G$ is $K_4$-free,
\item $X$ is an independent set,
\item $G-X$ has a homomorphism to $C_k$, in particular it is 3-colorable and triangle-free.
\end{enumerate}
\end{observation}
For a graph $G$, a function $\wei : E(G) \to \N \cup \{0\}$, and a set $E' \subseteq E(G)$,
we define $\wei(E') := \sum_{e \in E'} \wei(e)$.

Consider a certain variant of the \textsc{Maximum Weight Matching} problem.
An instance of \textsc{Maximum Weight Matching$^*$} (\MWM) is a tuple $(G,U,\wei,k)$, where $G$ is a graph,
$U$ is a subset of its vertices, $\wei : E(G) \to \N \cup \{0\}$ is an edge weight function, and $k$ is an integer.
We ask whether $G$ has a matching $M$, such that $\wei(M)\geq k$ and $M$ covers all vertices from $U$.
By a simple reduction to the \textsc{Maximum Weight Matching} problem we show that \MWM can be solved in polynomial time.

\begin{restatable}{lemma}{perfectmatching}
\label{lem:perfect-matching}
The \MWM problem can be solved in polynomial time.
\end{restatable}

The proof of \cref{lem:perfect-matching} is rather standard and thus we defer it to the appendix.

\newpage
\section{Reduction from \colorext{W_5} to \itte in $S_{2,1,1}$-free graphs}\label{sec:reduction}
In this section we show that in $S_{2,1,1}$-free graphs the \colorext{W_5} problem can be reduced to a variant of the problem of finding an independent set intersecting all triangles.
We start with the analysis of $S_{2,1,1}$-free instances of \colorext{C_5}: recall that these are the graphs obtained from a yes-instance of \colorext{W_5} by removing an independent set mapped to the central vertex of $W_5$.

\subsection{\colorext{C_5} in $\{S_{2,1,1},K_3\}$-free graphs}
A connected bipartite graph is called {\it almost complete} if it is complete or can be obtained from a complete bipartite graph by removing the edges of some matching. 

\begin{observation}\label{obs:s211-triangle-free-structure}
Let $G$ be a connected $\{S_{2,1,1},K_3\}$-free graph.
Then either (i) $G$ is a path, or (ii) $G$ a cycle of length at least 5, or (iii) $G$ is an almost complete bipartite graph.
\end{observation}
\begin{proof}
We consider the following cases.
\paragraph{Case 1: $G$ does not have cycles. } Hence, if $G$ is not a path, it has a vertex $v$ of degree at least 3.
However, if now there exists a vertex that is non-adjacent to $v$, then, as $G$ is connected, there exists a vertex $u$ at distance 2 from $v$. 
If $w$ is the common neighbor of $u$ and $v$, and $v_1, v_2$ are some other two neighbors of $v$ (that are non-adjacent to $u$ as $G$ does not have cycles), then $\{v,v_1,v_2,u,w\}$ induces a copy of $S_{2,1,1}$. 
Hence, $G$ is a tree that has a universal vertex $v$, so $G$ must be a star, i.e., an almost complete bipartite graph.

\paragraph{Case 2: $G$ is non-bipartite or $C_4$-free.} If $G$ is bipartite, let $k\geq 6$ be the length of a shortest induced cycle in $G$.
Such a cycle must exist, since $G$ is $C_4$-free and contains a cycle of length at least 6, so also an induced cycle of length at least 6. 
If $G$ is non-bipartite, let $k$ be the length of a shortest induced odd cycle in $G$.
Let $C$ be an induced cycle of length $k$ in $G$, clearly $k \geq 5$.
Denote by $v_1,v_2,\ldots,v_k$ the consecutive vertices of $C$.

We claim that $V(C)=V(G)$.
Suppose there exists a vertex in $G$ that does not belong to $C$. Since $G$ is connected, some vertex, say $v_1$, of $C$ has a neighbor, say $u$, which is not in $C$. 
Since $G$ is $K_3$-free and $C$ is an induced cycle of length $k \geq 5$, the graph induced by the set of vertices $\{u, v_1, v_2, v_3, v_k\}$ is a copy of $S_{2,1,1}$ unless $uv_3$ is an edge in $G$. 
So since now assume that $uv_3$ is an edge in $G$. Then in particular $\{u,v_1,v_2,v_3\}$ induces a cycle $C_4$. Hence if $G$ is bipartite, we get a contradiction and in this case we conclude that $G$ is a cycle of length at least $6$.
If $G$ is not bipartite, we similarly show that $uv_{k-1}$ is an edge in G. Thus, $u, v_3, v_4, \ldots, v_{k-1}$ are consecutive vertices of a cycle of length $k-2$ in $G$. 
We found an odd cycle in $G$ of length smaller than $k$. 
Consequently, there is an induced odd cycle in $G$ which is shorter than the cycle $C$, a contradiction. 
Thus, $G$ is an odd cycle of length at least $5$, and that completes the proof.

\paragraph{Case 3: $G$ is bipartite and contains a cycle $C_4$.} Let us denote the bipartition classes of $G$ by $V$ and $U$ and let $C$ be a copy of $C_4$ contained in $G$. Denote by  $v_1,v_2$ (respectively, $u_1,u_2$) the vertices of $C$ in $V$ (resp. in $U$). 

Suppose there is a vertex in $G$ which is not a neighbor of any vertex in $C$. By connectivity of $G$ there is a vertex, say $x$, at distance $2$ from $C$. We can assume without loss of generality that $v_1$ is a vertex in $C$ at distance $2$ from $x$. Let $y$ be a common neighbor of $v_1$ and $x$. It is easy to see that the graph induced by the set of vertices $\{x,y,v_1,u_1,u_2\}$ is a copy of $S_{2,1,1}$, a contradiction. Thus, every vertex in $G$ is a neighbor of some vertex of $C$. Consequently, every vertex of $U$ (resp. $V$) is adjacent to $v_1$ or $v_2$ (resp. $u_1$ or $u_2$).

Assume now that $v_1$ has at least two neighbors, say $w_1,w_2$ which are not neighbors of $v_2$. Then the set of vertices $\{v_2,u_1,v_1,w_1,w_2\}$ induces a copy of $S_{2,1,1}$, a contradiction. Thus, there is at most one vertex which is adjacent to $v_1$ but not to $v_2$. We denote this vertex by $u_1'$ if it exists. Similarly, there is at most one vertex adjacent to $v_2$ (resp. $u_1,u_2$) but not to $v_1$ (resp. $u_2,u_1$). We denote this vertex by $u_2'$ (resp. $v_1',v_2'$) if it exists. 
We have shown that each vertex $v_1,v_2$ (resp. $u_1,u_2$) is adjacent to all but at most one vertex of $U$ (resp. $V$). 

Suppose now that $u_1'$ is non-adjacent not only to $v_2$ but to some other vertex $x\in V$ as well. Clearly, at least one of the vertices $u_1,u_2$ is adjacent to $x$ because otherwise $x$ is not a neighbor of any vertex of $C$. Assume without loss of generality that $u_1$ is adjacent to $x$. Since $u_1'$ is adjacent to $v_1$, the set of vertices $\{u_1',v_1,u_1,v_2,x\}$ induces a copy of $S_{2,1,1}$, a contradiction. Thus, $u_1'$ is adjacent to all vertices of $V-\{v_2\}$. We similarly show that $u_2'$ (resp. $v_1',v_2'$) is adjacent to all vertices of $V-\{ v_1\}$ (resp. $U-\{ u_2\}, U-\{ u_1\}$). 

Finally, suppose that some vertex $v\in V-\{ v_1,v_2,v_1',v_2'\}$ is non-adjacent to some two vertices, say $x_1,x_2$ of $U$. Then, the set of vertices $\{ v,u_1,v_1,x_1,x_2\}$ induces a copy of $S_{2,1,1}$, a contradiction. Thus, every vertex of $V-\{ v_1,v_2,v_1',v_2'\}$ is adjacent to all vertices of $U$ except possibly one. We similarly show that every vertex of $U-\{ u_1,u_2,u_1',u_2'\}$ is adjacent to all vertices of $V$ except possibly one. Thus, $G$ is an almost complete bipartite graph. 
\end{proof}

We denote the consecutive vertices of $C_5$ by $1,2,3,4,5$. We will refer to the vertices of $C_5$ as colors. We will say that two colors are {\it neighbors} if they are neighbors on the cycle $C_5$. 

Let $G$ be a graph, let $W \subseteq V(G)$, and let $\varphi:W \to V(C_5)$ be a coloring of vertices of $W$. 
We say that a pair of vertices $\{u,v\}\subseteq V(G)$ is \emph{conflicted} if $u,v \in W$ and there is an $u$-$v$ path $P$ of length at most $3$ such that $W\cap V(P)=\{ u,v\}$ and $\varphi|_{\{ u,v\}}$ cannot be extended to a homomorphism from $P$ to $C_5$.
Equivalently, a pair $\{u,v\}\subseteq W$ of vertices of $G$ is conflicted in a coloring $\varphi$ of $G$ if and only if 
\begin{enumerate}[(i)]
\item $v \in N_G(u)$, and $\varphi(v)$ is non-adjacent to $\varphi(u)$ in $C_5$, or
\item there exists a path $u,w,v$ in $G$ with $w\not\in W$, and $\varphi(v)$ is adjacent to $\varphi(u)$, or
\item there exists a path $u,w_1,w_2,v$ in $G$ with $w_1,w_2\not\in W$, and $\varphi(u)=\varphi(v)$.
\end{enumerate}
We say that $\varphi$ is \emph{conflict-free}, if there is no pair of conflicted vertices in $G$ (with respect to $\varphi$).

Clearly, being triangle-free and conflict-free are necessary conditions to be a yes-instance of \colorext{C_5}.
The following lemma say that in $S_{2,1,1}$-free graphs they are also sufficient.

\begin{lemma}\label{obs:pre-ext-triangle-free}
Let $G$ be an $\{S_{2,1,1},K_3\}$-free graph, let $W \subseteq V(G)$, and let $\varphi:W \to V(C_5)$. 
If $\varphi$ is conflict-free, then $\varphi$ can be extended to a homomorphism $\psi:G \to C_5$.
\end{lemma}
\begin{proof} 
Clearly, by the definition of a conflict-free coloring, it suffices to show that every conflict-free coloring can be extended to a conflict-free coloring of the whole graph $G$. Assume otherwise, i.e. that there is a conflict-free coloring $\varphi:W \to V(C_5)$ with $W\not=V(G)$ which cannot be extended to a conflict-free coloring of $G$. We can assume without loss of generality that $\varphi$ is maximal with respect to this property, i.e. it is not possible to color any uncolored vertex of $G$ so that the resulting coloring is still conflict-free. 

By \cref{obs:s211-triangle-free-structure}, $G$ is a cycle, a path or an almost complete bipartite graph. We consider two cases. 

\paragraph{Case 1: $G$ is a cycle or a path.} Clearly, we can assume that at least one vertex of $G$ is colored. 
By maximality of the coloring $\varphi$, if $G$ is a path, then each end $v$ of this path is colored. Indeed, assume that $v$ does not have a color and let $u$ be the closest colored vertex. Depending on the distance from $v$ to $u$ we can always choose a color for $v$ so that the pair $\{ u,v\}$ is not conflicted. 

We proved that any uncolored vertex of $G$ is an internal vertex of a path in $G$ whose only colored vertices are the ends of this path. Consider such a path $P$ of the largest length and let $u$ and $v$ be the colored ends of this path. Let $\alpha$ (resp. $\beta$) be the color of $u$ (resp. $v$) in the coloring $\varphi$. Denote by $w$ the neighbor of $v$ on $P$. If $P$ has more than $2$ internal vertices then we color $w$ with a color which is a neighbor of the color $\beta$ and is different from the color $\alpha$. Obviously, such a color exists. The resulting extension of the coloring $\varphi$ is not conflicted, because $\{ u,w\}$ is not a conflicted pair in this extension, a contradiction with maximality of $\varphi$. 
If $P$ has at most $2$ internal vertices, then by the fact that $u,v$ are not conflicted and the definition of conflicted pair, $\vphi|_{\{u,v\}}$ can be extended to a homomorphism from $P$ to $C_5$. Therefore $\vphi$ can be extended to internal vertices of $P$ which contradicts the maximality of $\vphi$.
We have shown the lemma if $G$ is a path or a cycle.

\paragraph{Case 2: $G$ is an almost complete bipartite graph. } We will prove first the following statement. 
\begin{claim}\label{clm:conflictfree}
A coloring $\varphi:W \to V(C_5)$ of an almost complete bipartite graph $G$ is conflict-free if and only if the following conditions are satisfied 
\begin{enumerate}[(i)]
\item $\varphi$ is a $C_5$-coloring of the induced graph $G[W]$, and 
\item if some two vertices $u,v\in W$ in the same bipartition class have a common neighbor in $G$ then the colors of $u$ and $v$ are not neighbors, and
\item vertices in different bipartition classes have different colors.
\end{enumerate}
\end{claim}
\begin{claimproof}
First assume that that $\varphi$ is conflict-free. If some two adjacent vertices in $G[W]$ are not colored with colors which are neighbors then these vertices form a conflicted pair, a contradiction. Suppose that some two vertices $u,v\in W$ are in the same bipartition class and have a common neighbor, say $w$, in $G$. If the colors of $u$ and $v$ are neighbors and $w\not\in W$ then the pair $\{ u,v\}$ is conflicted by condition (ii) in the definition of a conflicted pair, a contradiction. If the colors of $u$ and $v$ are neighbors and $w\in W$ then the pair $\{ u,w\}$ or $\{ w,v\}$ is conflicted by condition (i) in the definition of a conflicted pair, a contradiction again. Finally, assume that some two vertices $u,v$ in different bipartition classes have the same color. If $u$ and $v$ are joined by an edge then the pair $\{ u,v\}$ is conflicted. Otherwise there is a path $u,w_1,w_2,v$ in $G$. In this case the pair $\{ u,v\}$ is conflicted by the condition (iii) or (ii) or (i) in the definition of a conflicted pair, a contradiction. We proved that the conditions (i)-(iii) are satisfied. 

Conversely, assume that the conditions (i)-(iii)  are satisfied for a coloring $\varphi$. One can readily verify that condition (i) (resp. (ii) and (iii))  implies that condition (i) (resp. (ii) and (iii)) in the definition of a conflicted pair is not satisfied. Thus, $\varphi$ is conflict-free which completes the proof of the claim. 
\end{claimproof}

Let $U$ and $V$ be the bipartition classes of the graph $G$. To prove the lemma for almost complete bipartite graphs we consider two subcases. 

\subparagraph{Subcase (i): Each two vertices $u,v\in V(G)$ in the same bipartition class have a common neighbor. } 
By \cref{clm:conflictfree} vertices of $V\cap W$ (resp. $U\cap W$) are colored in $\varphi$ with at most two colors and if both colors are used then these colors are not neighbors. Moreover, the colors used on vertices of $V\cap W$ are different from the colors used on vertices of $U\cap W$. Thus, without loss of generality, we can assume that the colors of vertices of $V\cap W$ are elements of the set $\{ 1,3\}$ and the colors of vertices of $U\cap W$ are elements of the set $\{ 2,4\}$. We observe that if both colors $1$ and $4$ are used by the coloring $\varphi$ then each of these colors is used on one vertex only and the vertices colored with these colors are non-adjacent in $G$. Otherwise the coloring $\varphi$ is not conflict-free. If some vertex of $V$ (resp. $U$) is not colored in the coloring $\varphi$, then we color it with color $3$ (resp. $2$). By \cref{clm:conflictfree} this extension of $\varphi$ is conflict-free, a contradiction with maximality of $\varphi$. 

\subparagraph{Subcase (ii): Some two vertices $u,v\in V(G)$ in the same bipartition class do not have a common neighbor. } 
We observe that if both bipartition classes of an almost complete bipartite graph $G$ have at least three vertices or $G$ is a star then 
each two vertices $u,v\in W$ in the same bipartition class have a common neighbor in $G$. So, $G$ has exactly two vertices in some bipartition class, say $V$, and has at least three vertices in the bipartition class $U$ (recall that $G$ is connected). Moreover, if $G$ is obtained from a complete bipartite graph by removing only one edge then the condition in subcase (ii) is not satisfied either. Thus, $G$ is obtained from a complete bipartite graph with $|V|=2$ and $|U|\geq 3$ by removing a $2$-edge matching. This graph has exactly two vertices, say $u,v\in U$, of degree $1$. This is the only pair of vertices in $G$ which does not have a common neighbor. If, in the coloring $\varphi$, $u$ or $v$ is not colored or the colors of $u$ and $v$ are not neighbors, then we proceed as in subcase (i). However, it may happen that in the coloring $\varphi$ the colors of the vertices $u$ and $v$ are neighbors, say $1$ and $2$, respectively. Let $u'$ (resp.$v'$) be the neighbor of $u$ (resp. $v$) in $G$. It follows from \cref{clm:conflictfree} that if $\varphi(u)=1$ and $\varphi(v)=2$ then $\varphi$ is conflict-free if and only if $\varphi(x)=4$ for every $x\in (W-\{ u,v\})\cap U$ (see condition (ii) in \cref{clm:conflictfree}), $\varphi(u')=5$ if $u'\in W$ and $\varphi(v')=3$ if $v'\in W$ (see conditions (i) and (iii) in \cref{clm:conflictfree}). Thus, the conflict-free coloring $\varphi$ can be extended to a conflict-free coloring of $G$. This contradiction completes the proof. 

\end{proof}

\subsection{The \textsc{\textbf{Independent Triangle Transversal Extension}} problem}
The following auxiliary problem plays a crucial role in our algorithm.

\problemdef%
{\textsc{Independent Triangle Transversal Extension} (\itte)}%
{A graph $G$, sets $X',Y' \subseteq V(G)$, and $E' \subseteq E(G)$}%
{Is there an independent set $X \subseteq V(G)$, such that (i) $X' \subseteq X$, (ii) for every $e\in E'$ it holds that $e \cap X \neq \emptyset$, (iii) $Y' \cap X = \emptyset$, (iv) $G-X$ is triangle-free?}


We observe that if $G$ contains a $K_4$, then at most one vertex of such a clique can belong to an independent set, and the graph induced by the remaining part is not triangle-free. 
Thus we immediately obtain the following observation which will be used throughout the paper implicitly.

\begin{observation}\label{obs:no-k4}
If $(G,X',Y',E')$ is an instance of \itte such that $G$ is not $K_4$-free, then it is a no-instance.
\end{observation}

As the final result of this section we show that \colorext{W_5} in $S_{2,1,1}$-free graphs
can be reduced in polynomial time to \itte in $S_{2,1,1}$-free graphs.

\begin{theorem}\label{thm:ittereduce}
The \colorext{W_5} in $S_{2,1,1}$-free graphs can be reduced in polynomial time to the \itte problem in $S_{2,1,1}$-free graphs.
\end{theorem}

\begin{proof}
Let $(G,U,\vphi)$ be an instance of \colorext{W_5}. Note that we can assume that there are no vertices $u,v\in U$ such that $uv\in E(G)$ and $\vphi(u)\vphi(v)\notin E(W_5)$, otherwise we immediately report a no-instance (formally, we can return a trivial no-instance of \itte, e.g., $(K_4,\emptyset,\emptyset,\emptyset)$). This can be verified in polynomial time. 

We define the instance $(G,X',Y',E')$ of \itte as follows. We initialize $X'$, $Y'$ and $E'$ as empty sets. We add to $X'$ every vertex precolored with $0$ and we add every vertex precolored with vertex other than $0$ to $Y'$. 
To construct $E'$, consider each pair $u,v \in U$ of conflicted vertices with respect to $\varphi: Y' \to [5]$, and let $P$ be a witness of $u$ and $v$.
By our first assumption, $|P|\geq 2$. 
If the consecutive vertices of $P$ are $u,w,v$, then we add $w$ to $X'$. 
If the consecutive vertices of $P$ are $u,w_1,w_2,v$, then we add the edge $w_1w_2$ to $E'$. This completes the construction of the instance $(G,X',Y',E')$ of \itte. Clearly the reduction is done in polynomial time.

Let us verify that the instance $(G,X',Y',E')$ of \itte is equivalent to the instance $(G,U,\vphi)$ of \colorext{W_5}. First assume that there is a set $X\subseteq V(G)$ that is a solution to the instance $(G,X',Y',E')$ of \itte. Then $G-X$ is triangle-free. Suppose that there is a conflicted pair of vertices $u,v$ in $G-X$ and let $P$ be the path such that the precoloring of $u,v$ cannot be extended to $P$. Observe that by the construction of $E'$ at least one vertex of $P$ must be in $X$, a contradiction. Hence, by calling \cref{obs:pre-ext-triangle-free} on $G-X$, we conclude that the precoloring of $G-X$ can be extended to all vertices of $G-X$ using only colors $1,2,3,4,5$, and then extended to whole $G$ by coloring every vertex of $X$ with $0$.

So now suppose that $(G,U,\vphi)$ is a yes-instance of \colorext{W_5}. Then there exists a $W_5$-coloring $\psi$ of $G$ that extends $\vphi$. Define $X:=\psi^{-1}(0)$. Let us verify that $X$ satisfies the desired properties. If follows from the definition that $X$ is an independent set. Suppose that $G-X$ contains a triangle. Then $G-X \not\to C_5$, and thus the vertices of $G-X$ cannot be colored using only colors $1,2,3,4,5$, a contradiction. Consider a vertex $x\in X'$. Then either $\varphi(x)=0$, or there is a path with consecutive vertices $u,x,v$ with $u,v \in U$, such that $\varphi$ cannot be extended to $x$ using only colors $1,2,3,4,5$. Therefore in both cases $\psi(x)=0$, so $x\in X$. Now consider $y\in Y'$. Then $\varphi(y)\neq 0$, so $y\notin \psi^{-1}(0)=X$. Finally, consider an edge $xy\in E'$. Then there is a path with consecutive vertices $u,x,y,v$ with $u,v \in U$, so that $\varphi$ cannot be extended to $x,y$ using only colors $1,2,3,4,5$. We conclude that one of $x,y$ is mapped by $\psi$ to $0$, and thus one of $x,y$ is in $X$. That completes the proof.
\end{proof}

\newpage
\section{Solving \itte: basic toolbox}\label{sec:tools}
In this seciton we provide some basic tools to deal with \itte.
In particular, we prove that the problem behaves well under standard graph decompositions: modular decomposition, clique-cutset decomposition, and tree decomposition.

First, let us prove that in order to solve \itte it is sufficient to solve instances $(G,X',Y',E')$ of \itte where $Y' = \emptyset$.

\begin{proposition}\label{prop:itt-itte}
Let $(G,X',Y',E')$ be an instance of \itte. We can construct in polynomial time an instance $(\tG,\tX',\emptyset,\tE')$ of \itte that is equivalent to $(G,X',Y',E')$ and such that $\tG$ is an induced subgraph of $G$.
\end{proposition}

\begin{proof}
We initialize $\tX':=X'$ and $\tE':=E'$. If there is an edge $uv\in E'$ with $u,v\in Y'$, then we conclude that we are dealing with a no-instance. If there is an edge $uv\in E'$ with $u\in Y'$ and $v\notin Y'$, then we add $v$ to $\tX'$. Consider a triangle $uvw$. If $u,v,w\in Y'$, we report a no-instance. If $u,v\in Y'$ and $w\notin Y'$, we add $w$ to $\tX'$. If $u\in Y'$ and $v,w\notin Y'$, then we add the edge $vw$ to $\tE'$. Finally, we set $\tG=G - Y'$. 
This completes the construction of the instance $(\tG,\tX',\emptyset,\tE')$. Clearly the construction is performed in polynomial time.

Let us verify the equivalence. First, assume that $(G,X',Y',E')$ is a yes-instance and $X$ is a solution. Let us show that $X$ is also a solution for $(\tG,\tX',\emptyset,\tE')$. Consider a vertex $v\in \tX'$. If $v\in X'$, then $v\in X$ since $X$ is a solution for $(G,X',Y',E')$. If there is an edge $uv\in E'$ with $u\in Y'$, then it must hold that $v\in X$ since we cannot have $u\in X$. If there is a triangle $uvw$ with $u,w\in Y'$, then $y\in X$ since $G-X$ is triangle-free and $u,w$ cannot be in $X$. Now let $uv\in \tE'$. If $uv\in E'$, then clearly one of $u,v$ must be in $X$. If there is a triangle $uvw$ with $w\in Y'$, then one of $u,v$ must be in $X$ since $G-X$ is triangle-free. Finally, observe that if $G-X$ is triangle-free, then $\tG-X$, as an induced subgraph of $G-X$, is also triangle-free.

So now assume that there is a solution $X$ for $(\tG,\tX',\emptyset,\tE')$. We will verify that $X$ is also a solution for  $(G,X',Y',E')$. First, observe that since $X\subseteq V(\tG)$, we have that $X\cap Y'=\emptyset$. Consider $u\in X'$. Note that $X'\subseteq \tX'$, so $u\in X$. Now consider $uv\in E'$. If $u,v\notin Y'$, then $uv \in \tE'$ and one of $u,v$ is in $X$. If $u\in Y'$, then $v\in \tX'$ and thus $v\in X$. It remains to verify that $G-X$ is triangle-free. Suppose to the contrary that there is a triangle $uvw$ in $G-X$. Since $\tG-X$ is triangle-free, then at least one of $u,v,w$, say $u$, must be in $Y'$. Note that if all vertices $u,v,w\in Y'$, then we already reported a no-instance. Therefore, either we added an edge $vw$ to $\tE'$ and one of $v,w$ is in $X$, or we added one of $v,w$, say $v$, to $\tX'$, and thus $v\in X$, a contradiction.
\end{proof}

\subsection{Modular decomposition and prime graphs}
Let $G$ be a graph and let $M$ be a proper subset of $V(G)$. We say that $M$ is a \emph{module} of $G$ if for every vertex $v\in V(G) \setminus M$ either $v$ is adjacent to every vertex of $M$ or is non-adjacent to every vertex of $M$. We say that a module $M$ is \emph{non-trivial} if $|M|>1$, otherwise $M$ is \emph{trivial}. We say that $G$ is \emph{prime} if every module of $G$ is trivial. 

Let $\cM$ be a partition of $V(G)$ into modules. We define the \emph{quotient} graph $Q(\cM)$ to be a graph such that there is a bijection between the modules in $\cM$ and vertices of $Q(\cM)$, and for vertices $m,m'$ of $Q(\cM)$, there is an edge $mm'$ if and only if the corresponding modules $M,M'$ are adjacent (i.e., all edges between $M$ and $M'$ are present). Observe that $Q(\cM)$ is an induced subgraph of $G$ since it can be obtained by picking exactly one vertex from each module in $\cM$.

A classic result of Gallai~\cite{gallai} asserts that if $G$ and $\overline{G}$ are both connected, then the partition $\cM$ of $V(G)$ into inclusion-wise maximal modules is uniquely determined and can be found in polynomial time. In this case the quotient graph $Q(\cM)$ is prime.

In the next lemma we show that \itte combines well with modular decompositions of graphs.

\begin{lemma}\label{lem:X-and-X-star}
Let $\cX$ be a hereditary class of graphs and let $\cX^*$ be the class of all induced subgraphs of the graphs in $\cX$ that are either prime or cliques. If \itte can be solved in polynomial time on $\cX^*$, then it can be solved in polynomial time on $\cX$.
\end{lemma} 

\begin{proof}
Let $(G,X',Y',E')$ be an instance of \itte such that $G\in \cX$. By \cref{prop:itt-itte}, we can assume that $Y'=\emptyset$. Note that by \cref{prop:itt-itte}, we obtain an instance that is still in $\cX$ since the class is hereditary. Clearly, if $G$ contains only one vertex, the problem is trivial and can be solved in constant time. If $G$ is disconnected, then we can solve the problem on every instance separately. If $\overline{G}$ is disconnected, then observe that the vertex set of a connected component of $\overline{G}$ is a module. In this case we define the partition $\cM$ of $V(G)$ into modules so that the modules are the vertex sets of connected components in $\overline{G}$. So now assume that both $G$ and $\overline{G}$ are connected and $G$ contains at least two vertices. In this case we define $\cM$ as a partition of $V(G)$ into maximal modules -- recall that if $G$ and $\overline{G}$ are connected, then $\cM$ is uniquely determined and can be found in polynomial time~\cite{gallai}. 

Let $Q:=Q(\cM)$ be the quotient graph of $G$. Note that $Q$ contains at least two vertices. If $\overline{G}$ is disconnected, then $Q$ is a clique and if $\overline{G}$ is connected, then $Q$ is prime. Hence, we have that $Q\in \cX^*$. For a module $M$ of $\cM$ we will denote by $m$ the vertex representing $M$ in the graph $Q$. 

We will construct an instance $(Q,\tX',\emptyset,\tE')$ of \itte equivalent to the instance $(G,X',\emptyset,E')$.

Since there are at least two vertices in $G$ and $G$ is connected, the graph $Q$ is also connected and therefore there are no isolated vertices in $Q$. We can conclude that no module corresponding to a vertex of $Q$ contains a triangle since, by \cref{obs:no-k4}, $G$ is $K_4$-free.

\paragraph{Types of modules.} We distinguish three types of modules: (1) independent sets, (2) bipartite with at least one edge, and (3) non-bipartite. We observe that there are no edges in $G$ joining any two modules of type (2) or (3) because $G$ is $K_4$-free. Consequently, in the graph $Q$ the vertices representing modules of type (2) or (3) are adjacent only to vertices representing modules of type (1).

For a given instance $(G,X',\emptyset,E')$, we distinguish two subtypes of modules of type (2). A module $M$ of type (2) is of type (2b) if there is a bipartition class in $G[M]$ which contains $X'\cap M$ and a vertex of every edge of $E'$ intersecting $M$. A module of type (2) is of type (2a) if it is not of type (2b).

Note that in any solution $X$ of the \itte problem, for a module $M$ of type (3), we have $X\cap M=\emptyset$. Suppose to the contrary that $X\cap M\neq \emptyset$. Then for any module $M'$ adjacent to $M$, we have $M'\cap X=\emptyset$. Moreover, since $G[M]$ is non-bipartite, there is an edge $uv$ in $G[M]-X$. Therefore, for any vertex $w$ in an adjacent module, $uvw$ is a triangle in $G-X$, a contradiction. Similarly, we can observe that if for a module $M$ of type (2) we have $M\cap X\neq \emptyset$, then one bipartition class of $G[M]$ is fully taken to $X$. So if $M$ of type (3) and $X'\cap M \neq \emptyset$, or $M$ of type (2) and there is a connected component $C$ of $G[M']$ with two vertices of $X'$ in distinct bipartition classes, then we can report a no-instance.

\paragraph{Construction of the equivalent instance.} The set $\tX'$ consists of the vertices $m$ of $Q$ such that the corresponding module $M$ of $G$ satisfies at least one of the following conditions:
\begin{enumerate}
\item $M$ contains some vertex of $X'$, 
\item $G[M]$ contains some edge of $E'$,
\item some neighbor of $M$ is of type (3), 
\item some neighbor of $M$ is of type (2a).
\end{enumerate}

The set $\tE'$ consists of the edges $mm'$ of $Q$ such that for the corresponding modules $M, M'$ of $G$ satisfy at least one of the following conditions:
\begin{enumerate}
\item both $M$ and $M'$ are of type (1) and there is an edge in $E'$ joining a vertex in $M$ with a vertex in $M'$, 
\item one of the modules $M$ or $M'$ is of type (2b). 
\end{enumerate}

This completes the construction of $(Q,\tX',\emptyset,\tE')$.


\paragraph{Correctness.} Let us verify the correctness, i.e., that the created instance $(Q,\tX',\emptyset,\tE')$ is equivalent to the instance $(G,X',\emptyset,E')$. 

First assume that $X$ is a solution for the instance $(G,X',\emptyset,E')$. We will prove that $\tX=\{ m\in V(Q): M\cap X\not=\emptyset\}$ is a solution for the instance $(Q,\tX',\emptyset,\tE')$. Since $X$ is an independent set of vertices of $G$, $\tX$ is an independent set of vertices of $Q$. Suppose the set $\tX$ does not intersect some triangle $m_1m_2m_3$ in $Q$. Then each triangle in $G$ which has one vertex in $M_1$, another in $M_2$ and the last one $M_3$ has no vertex in $X$, a contradiction. 

Let us show now that $\tX'\subseteq\tX$. Obviously, every module $M$ containing a vertex of $X'$ or an edge of $E'$ has a non-empty intersection with $X$, so the corresponding vertex $m$ in $Q$ is in $\tX$. Consider now a module $M$ whose neighbor, say $M'$, is of type (3) or (2a) and suppose no vertex of $M$ is in $X$. Then, every triangle with one vertex in $M$ and the other two in $M'$  has a common vertex with $X$ in $M'$. Moreover, all modules which are neighbors of $M'$ contain no vertex of $X$, because $X$ is independent. The set $X\cap M'$ must be an independent set of vertices intersecting every edge of $G[M]$, otherwise vertices of some edge in $M'$ together with any vertex of $M$ form a triangle with no vertex in $X$. This is, however, impossible if $M'$ is of type (3) because $G[M']$ is a non-bipartite graph, a contradiction. Hence, $X$ intersects the module $M$ if $M'$ is a module of type (3). If $M'$ is of type (2a) then $X\cap M'$ must be a bipartition class of $M'$ that contains all vertices of $X' \cap M'$. By the definition of a module of type (2a), there is an edge $e\in E'$ with an endvertex in $M'$ and no endvertex in $X\cap M'$. Thus, the other endvertex of $e$ is a member of some module which is a neighbor of $M'$. We have already observed that no vertex of $X$ is a member of such module. Thus, $e\in E'$ is disjoint from $X$, a contradiction. Hence, $X$ intersects the module $M$ if $M'$ is a module of type (2a). We have shown that $\tX'\subseteq\tX$.

Let us prove that $\tX$ intersects every edge of $\tE'$. 
Consider an edge $mm' \in \tE'$. If the corresponding modules $M$ and $M'$ are of type (1) and there is an edge in $E'$ joining a vertex in $M$ with a vertex in $M'$, then one of the ends of this edge is in $X$, so $X$ intersects $M$ or $M'$ and, consequently, $m\in \tX$ or $m'\in \tX$. If one of the modules $M$ or $M'$, say $M'$, is of type (2b) then vertices of any edge in $M'$ together with a vertex of $M$ form a triangle, so $X$ contains a vertex of $M$ or $M'$. Consequently, one of the ends of the edge $mm'\in\tE'$ is in $\tX$. We have shown that $\tX$ intersects each edge of $\tE'$ so $\tX$ is indeed a solution for the instance $(Q,\tX',\emptyset,\tE')$. 

Assume now that $\tX$ is a solution for the instance $(Q,\tX',\emptyset,\tE')$ of \itte. We define $X$ to be the set consisting of all vertices of every module $M$ of type (1) such that $m\in \tX$ and all vertices of a bipartition class, say $B_{M'}$, of each module $M'$ of type (2b) such that $m'\in\tX$ and $B_{M'}$ intersects every edge of $E'$ with an endvertex in $M'$ and contains all vertices of the set $X'\cap M'$. Notice that by the definition of a module of type (2b) such a bipartition class exists. We will prove that $X$ is a solution for the instance $(G,X',\emptyset,E')$. 

Since $\tX$ is an independent set of vertices of $Q$, it is clear that $X$ is an independent set of vertices of $G$. As we have already observed, no module corresponding to a vertex of $Q$ contains a triangle. If a triangle $T$ in $G$ has two vertices $u,w$ in the same module $M$ and the third one, say $v$, in another module, say $M'$, then $M$ must be of type (2) or (3) and $M'$ must be of type (1). If $M$ is of type (2a) or (3) then $v \in X$ because, by the definition of $\tX'$, $m'\in\tX'\subseteq\tX$. If $M$ is of type (2b) then, by the definition of $\tE'$,  $mm'\in\tE'$, so $m'\in\tX$ or $m\in\tX$. By definition of $X$, in the former case, we have $v \in X$, as $M'$ is of type (1), and in the latter one, either $u \in X$ or $w \in X$. Last, if a triangle $T$ in $G$ has its vertices in three different modules $M,M',M''$ then these modules must be of type (1) because otherwise there is a copy of $K_4$ in $G$. Since the vertices $m,m',m''$ form a triangle in $Q$, one of these vertices, say $m$, is in $\tX$. Consequently the vertex of $T$ contained in $M$ is in $X$ because $M\subseteq X$. We have shown that every triangle in $G$ has a vertex in $X$.  

Let us prove now that $X$ intersects every edge $e\in E'$. If $e$ is contained in some module $M$ then $m\in\tX'\subseteq\tX$. Notice that $M$ is of type neither (3) nor (2a) because the vertices of $Q$ corresponding to all neighbors of modules of these types are in $\tX'$ and $\tX'$ is independent because $\tX'\subseteq\tX$. Thus, $M$ is of type (2b), so we are done by the definition of $X$. If $e\in E'$ joins two vertices of modules $M$ and $M'$ of type (1), then $mm'\in\tE'$ so $m\in\tX$ or  $m'\in\tX$ and we are done by the definition of $X$ again. If $e\in E'$ joins a vertex of a module $M$ of type (1) with a vertex of a module of type (3) or (2a) then $m\in\tX'\subseteq\tX$ and we are done again. Finally, if $e\in E'$ joins a vertex of a module $M$ of type (1) with a vertex of a module $M'$ of type (2b) then there is an edge $mm'\in\tE'$ so $m\in\tX$ or $m'\in\tX$. In both cases we conclude, by the definition of $X$, that $X$ intersects $e$. 

We shall finally prove that $X'\subseteq X$. We observe that no vertex of $X'$ is in a module of type (2a) or (3). Indeed, suppose to the contrary that $M\cap X'\neq\emptyset$ for a module $M$ of type (2a) or (3). Then the vertices in $Q$ corresponding to the modules adjacent to $M$ are in $\tX'$. Moreover, $m\in \tX'$ since $M$ contains a vertex from $X'$. Therefore the set $\tX'$ is not independent. Since $\tX$ contains $\tX'$, the set $\tX$ is not independent either. Thus, vertices of $X'$ are contained in modules $M$ of type (1) or (2b) such that $m\in\tX'\subseteq\tX$. By the definition of $X$, we have $X'\subseteq X$.

\paragraph{Running time.}  Classifying the modules, i.e., verifying if they are independent or bipartite, clearly can be done in polynomial time. Classifying modules of type (2) into subtypes, for each module $M$ of type (2) can be done in polynomial time. Indeed, it is enough to verify the condition for each connected component of $G[M]$ separately, and this can be done in polynomial time. Constructing the sets $\tX'$ and $\tE'$ also can be done in polynomial time. We solve the problem on $(Q,\tX',\emptyset,\tE')$ in polynomial time since $Q\in \cX^*$. Therefore the instance $(G,X',\emptyset,E')$ can be solved in polynomial time. That completes the proof.
\end{proof}

\subsection{Clique cutsets and atoms}
A (possibly empty) set $C \subseteq V(G)$ is a \emph{cutset} in $G$, if $V(G)-C$ is disconnected.
We say that $C$ is a \emph{clique cutset} if $C$ is a cutset and a clique.
We say that $G$ is an \emph{atom} if it does not contain clique cutsets. Note that, in particular, every atom is connected.
A triple $(A,C,B)$ is a \emph{clique cutset partition} if it is a partition of $V(G)$ such that $C$ is a clique cutset, there is no edge between $A$ and $B$, and $G[A \cup C]$ is an atom.

We will need the following.
\begin{theorem}[Tarjan~\cite{DBLP:journals/dm/Tarjan85}]\label{thm:clique-cutsets}
Let $G$ be a graph.
Then in polynomial time we can find a clique cutset partition or correctly conclude that $G$ is an atom.
\end{theorem}

We show that we can reduce \itte problem to instances $(G,X',Y',E')$ such that $G$ is an atom.
The algorithm is a standard recursion that exploits the existence of clique cutsets~\cite{DBLP:journals/dm/Tarjan85}.

\begin{lemma}\label{lem:atoms}
Let $\cX$ be a family of graphs, and let $\cX^a$ be the family of all atoms in $\cX$.
If \itte can be solved in polynomial time in $\cX^a$, then \itte can be solved in polynomial time in $\cX$.
\end{lemma}
\begin{proof}
Let $\cI=(G,X',Y',E')$ be an instance of \itte such that $G\in \cX$.
We compute the solution recursively.

If $G$ is an atom, $\cI$ can be solved in polynomial time by our assumption.
Otherwise, we find a clique cutset partition $(A,C,B)$ of $G$.
Define $G_{AC}:=G[A \cup C], G_{CB}:=G[C \cup B]$.
If $|C|>3$, or $|X' \cap C| \geq 2$, we report that $G$ is a no-instance. 
Otherwise, for every $C' \subseteq C - Y'$, such that $|C'|\leq 1$, we solve the instance $\cI_{C'}=(G_{AC},(X' \cap A)\cup C',Y' \cap A,E' \cap E(G_{AC}))$.
Let $M_A$ be the set of these sets $C'$ for which $\cI_{C'}$ is a yes-instance.
If $M_A=\emptyset$, we report that $G$ is a no-instance. 
If $|C|=2$, and $\emptyset \notin M_A$, we add the edge from $C$ to $E'$.
If $|C|=1$, and $\emptyset \notin M_A$, we add the vertex from $C$ to $X'$.
Then, we recursively solve the instance $\cI_B=(G_{CB}, X' - A, (Y'-A) \cup C_Y, E' \cap E(G_{CB}))$, where $C_Y=\{x \in C: \{x\} \notin M_A\}$.
\paragraph{Correctness.} First, we prove that if the algorithm returns that $\cI$ is a no-instance, then the solution cannot exist.
Indeed, it means that in some recursive call (i) $|C|>3$, or (ii) $|X' \cap C| \geq 2$, or (iii) $M_A=\emptyset$, or (iv) $\cI_B$ is a no-instance. 
In the first case, $G$ contains a copy of $K_4$, so by \cref{obs:no-k4}, $\cI$ is a no-instance.
In the second case, the set $X'$ is not independent, so again $\cI$ is a no-instance.
For the remaining cases, suppose that $\cI$ is a yes-instance, and $X$ is a solution to $\cI$.
However, then $X \cap (A \cup C)$ is a solution to $(G_{AC},X' \cap (A\cup C),Y' \cap (A\cup C),E' \cap E(G_{AC}))$, so $C' = X \cap C \in M_A$, a contradiction with (iii).
Similarly, $X \cap (B \cup C)$ is a solution to $\cI_B$, a contradiction with (iv).

Now observe that the solution returned by a recursive call for $\cI_B$ can be always extended to $G$.
Indeed, assume that $X$ is a solution to $\cI_B$, and let $C':=X \cap C$.
As $C$ is a clique, we must have $|C'|\leq 1$.
If $C'=\emptyset$, we must have that $|C| \leq 2$, as otherwise $C$ is triangle in $G-X$
However, then we have in particular that $E' \cap E(G[C])=\emptyset$ and $X' \cap C=\emptyset$, so if $|C|\leq 2$, then $\emptyset\in M_A$ and the solution can be extended to $G$.
On the other hand, if $C'=\{x\}$ for some $x \in C$, then $x \notin C_Y$, i.e., $\cI_{C'}$ is a yes-instance, so the solution also can be extended to $G$.
\paragraph{Running time.} Concluding that $G$ is an atom or finding the clique cutset partition $(A,B,C)$ can be done in polynomial time by~\cref{thm:clique-cutsets}. 
As $|C|\leq 3$, at every step we call the algorithm for instances that are atoms at most four times, so this step also takes polynomial time.
Since the set $A$ is non-empty, in every recursion call our instance is strictly smaller, so after at most $n$ calls, we obtain the solution.
Hence, the running time of the algorithm is polynomial in $|V(G)|$.
\end{proof}
\subsection{Bounded-treewidth graphs}
\emph{Monadic Second-Order Logic} ($\mathsf{MSO}_2$) over graphs consists of formulas with vertex variables, edge variables, vertex set variables, and edge set variables, quantifiers, and standard logic operators. We also have a predicate $\mathsf{inc}(v,e)$, indicating that the vertex $v$ belongs to the edge $e$.

For a graph $G$, let $\tw{G}$ denote the treewidth of $G$.
The classic result of Courcelle~\cite{COURCELLE199012} asserts that problems that can be expressed in $\mathsf{MSO}_2$ can be efficiently solved on graphs of bounded treewidth. The statement of the following version of Courcelle's theorem comes from Cygan et al.~\cite[Theorem 7.11]{DBLP:books/sp/CyganFKLMPPS15}.
\begin{theorem}[Courcelle~\cite{COURCELLE199012}]\label{thm:courselle}
Let $\Psi$ be a formula of $\textsf{MSO}_2$ and $G$ be an $n$-vertex graph equipped with evaluation of all the free variables of $\Psi$. 
Moreover, $G$ is given along with a tree decomposition of width $t$.
Then tverifying whether $\Psi$ is satisfied in $G$ can be done in time $f(||\Psi||, t) \cdot n$, for some computable function $f$.
\end{theorem}
It is straightforward to verify that \itte can be expressed as a formula of $\textsf{MSO}_2$ with free variables where the sets $X',Y'$, and $E'$ are free variables of $\Psi$. Furthermore, given a graph of bounded treewidth,
one can find a tree decomposition of bounded width in polynomial time, using one of FPT approximation algorithms for treewidth~\cite{DBLP:conf/focs/Korhonen21}. Summing up, we obtain the following.
\begin{corollary}\label{cor:itte-bounded-treewidth}
The \itte problem can be solved in polynomial time on graphs with bounded treewidth.
\end{corollary}

\newpage
\section{Solving \itte in claw-free graphs}\label{sec:clawfree}
Let $G$ be a connected graph. 
A \emph{strip structure of $G$} is a pair $(D, \eta)$, that consists of a simple graph $D$, a set $\eta(xy) \subseteq V(G)$ for every $xy \in E(D)$, and its non-empty subsets $\eta(xy, x), \eta(xy, y) \subseteq \eta(xy)$, satisfying the following conditions:
\begin{enumerate}
\item[(S1)] $|E(D)| \geq 3$, and there are no vertices of degree 2 in $D$,
\item[(S2)] the sets $\{\eta(e) : e \in E(D)\}$ form a partition of $V(G)$,
\item[(S3)] if $u \in \eta(e), v \in \eta(f)$, for some $e,f \in E(D)$, then $uv \in E(G)$ if and only if there exists $x \in V(D)$ such that $e$ and $f$ are incident to $x$, $u \in \eta(e, x)$ and $w \in \eta(f, x)$,
\item[(S4)] for every $x \in V(D)$, the set $\bigcup_{y \in N(x)} \eta(xy, x)$ induces a clique in $G$.
\end{enumerate}

Chudnovsky and Seymour~\cite{DBLP:journals/jct/ChudnovskyS08e} proved that every claw-free $G$ is either very simple
or admits a strip structure where the subgraphs induced by vertices assigned to a single edge are simple.
If $G$ is additionally of bounded maximum degree, then this result becomes particularly usuful.
The following corollary of the result of Chudnovsky and Seymour comes from the paper of Abrishami et al.~\cite[Corollary 3.5]{ACDR21}.

\begin{theorem}[Chudnovsky, Seymour~\cite{DBLP:journals/jct/ChudnovskyS08e}]\label{thm:strip-structure}
If $G$ is a connected claw-free graph with maximum degree at most $\Delta$, then either
$\tw{G} \leq 4\Delta + 3$, or $G$ admits a strip structure $(D, \eta)$ such that for every $e \in E(D)$, we have $\tw{G[\eta(e)]} \leq 4\Delta + 4$.
Moreover, $(D, \eta)$ can be found in time polynomial in $|V(G)|$.
\end{theorem}

The main result of this section is that \itte can be solved in polynomial time in claw-free graphs.

\begin{theorem}\label{thm:itt-claw-free}
\itte can be solved in polynomial time in claw-free graphs.
\end{theorem}
\begin{proof}
Let $(G,X',Y',E')$ be an instance of \itte, such that $G$ is claw-free.
By \cref{lem:atoms}, we can assume that $G$ is an atom, and, in particular, $G$ is connected.

First, we check whether $G$ contains $K_4$.
If so, by \cref{obs:no-k4}, we can safely return that we deal with a no-instance.
Hence, we assume that $G$ is $K_4$-free.
This implies that $\Delta(G) \leq 5$.
Indeed, if there exists a vertex $v$ of degree at least six, then $v$ either has three pairwise non-adjacent neighbors (and hence a claw), or three pairwise adjacent neighbors (and hence a $K_4$).

Since $G$ is claw-free and $\Delta(G) \leq 5$, by \Cref{thm:strip-structure}, either (i) $\tw{G} \leq 23$ or (ii) there exists a strip structure $(D,\eta)$ such that for every $e \in E(D)$ we have $\tw{G[\eta(e)]} \leq 24$.
By \cref{cor:itte-bounded-treewidth}, if (i) holds, it can be checked in polynomial time if $G$ is a yes-instance.
Therefore, we will assume that (ii) holds. 
By \Cref{thm:strip-structure} the strip structure $(D,\eta)$ can be computed in polynomial time.
As $G$ is connected, so is $D$.

Recall that by the definition of a strip structure, if $y_1,\ldots,y_d$ are neighbors of some $x \in V(D)$, the set $\bigcup_{i \in [d]}\eta(xy_i,x)$ induces a clique in $G$.
Hence, as $G$ is $K_4$-free, we have $\Delta(D) \leq 3$, which implies that the vertices of $D$ are either of degree 1 or 3.
Moreover, if $x$ is of degree 3, let $N(x)=\{y_1,y_2,y_3\}$ and note that since $G$ is $K_4$-free, $|\eta(xy_1,x)|=|\eta(xy_2,x)|=|\eta(xy_3,x)|=1$.
In this case, we denote the single member of $\eta(xy,x)$ by $v_{x,y}$. 

On the other hand, if $D$ contains a vertex $y$ of degree 1, two things can happen: either $D$ is isomorphic to $K_2$, or $y$ is adjacent to a vertex $x$ that is of degree 3 in $D$.
In the first case, if we denote by $e$ the edge of $D$, we have $\tw{G}=\tw{G[\eta(e)]}\leq 24$, so again, we can solve the problem in polynomial time by \cref{cor:itte-bounded-treewidth}.
In the second case, let $y_1,y_2$ be the other (than $y$) neighbors of $x$ in $D$.
We observe that $\{v_{x,y_1},v_{x,y_2}\}$ or $\{v_{x,y}\}$ is a clique cutset in $G$, a contradiction with $G$ being an atom.
Hence, we can assume that $D$ is 3-regular.

For $x \in V(D)$ and $y_1,y_2,y_3 \in N_D(x)$, let $E(x):=\{v_{x,y_1}v_{x,y_2},v_{x,y_2}v_{x,y_3},v_{x,y_3}v_{x,y_1}\}$. 
Recall that by the definition of a strip structure, $G[E(x)]$ is a triangle. 
The following reduction rule is straightforward.
\begin{redrule}\label{red:itt-claw-i}
If $E(x) \subseteq E'$, then $(G,X',Y',E')$ is a no-instance.
If $v_{x,y_1}v_{x,y_2},v_{x,y_2}v_{x,y_3} \in E',$ but $v_{x,y_3}v_{x,y_1} \notin E'$, we can safely add $v_{x,y_2}$ to $X'$ and remove $v_{x,y_1}v_{x,y_2},v_{x,y_2}v_{x,y_3}$ from $E'$.
Hence, for every $x \in V(D)$ it holds that $|E(x) \cap E'| \leq 1$. 
\end{redrule}

Consider an edge $xy$ of $D$.
Let $G_{xy}:=G[\eta(xy)]$, and let $I_{xy}=\{v_{x,y},v_{y,x}\}-Y'$.
Let $\cA(xy)$ be the (possibly empty) set of these sets $A$ which satisfy $X' \cap I_{xy} \subseteq A \subseteq I_{xy}$ and
\[\cI_A=(G_{xy}, (X' \cap V(G_{xy})) \cup A, (Y' \cap V(G_{xy})) \cup (I_{xy} - A), E' \cap E(G_{xy}))\]
 is a yes-instance of \itte. 
Observe that if there exists a solution $X$ of \itte for $(G,X',Y',E')$, then $X \cap \eta(xy)$ is a solution to $\cI_{A'}$ for $A'= \{v_{x,y},v_{y,x}\} \cap X$, and in this case $A' \in \cA(xy)$.
We obtain the following.

\begin{redrule}\label{red:itt-claw-ii}
If $\cA(xy)=\emptyset$ for some $xy \in E(D)$, then $(G,X',Y',E')$ is a no-instance. 
\end{redrule}

Our aim is to reduce the \itte problem to the \MWM problem.
Let $\cE:=\{xy \in E(D): \emptyset \notin \cA(x,y)\}$.
We construct an instance $(D',U,\wei,|\cE|)$ of \MWM as follows.
To obtain $D'$, we take a copy of the set $V(D)$, and for every $xy_1 \in E(D)$ we introduce an additional vertex $t_{xy}$.
Then, for every $xy \in E(D)$, we do the following (below we always use the notation $N_D(x)=\{y,y',y''\}$).
\begin{enumerate}
\item[(1)] if $\{v_{x,y},v_{y,x}\} \in \cA(xy)$, we add $xy$ to $E(D')$,
\item[(2)] if $\{v_{x,y}\} \in \cA(xy)$, we add $xt_{xy}$ to $E(D')$,
\item[(3)] if $\{v_{y,x}\} \in \cA(xy)$, we add $yt_{xy}$ to $E(D')$.
\item[(4)] if $\{v_{x,y'}v_{x,y''}\}=E' \cap E(x)$, we remove $xy$ and $xt_{xy}$ from $E(D')$ (if they were introduced).
\end{enumerate} 
Let $U = V(D) \subseteq V(D')$, i.e., the set of original vertices of $D$. 
For $xy \in E(D)$, we define $T(xy):=\{xy,xt_{xy},yt_{xy}\} \cap E(D')$.
We set $\wei:E(D') \to \{0,1\}$ to be $\wei(uv)=1$ if $uv \in T(xy)$ for some $xy \in \cE$, and $\wei(uv)=0$ otherwise. 
This concludes the construction of $(D',U,\wei,|\cE|)$.

By the definition of the strip structure, each edge $uv \in V(G)$ is either contained in $E(G_{xy})$, for some $xy \in E(D)$, or $u=v_{x,y'}$ and $v=v_{x,y''}$ for some $x \in U$.
In the first case, we say that $uv$ is of \emph{type I for $xy$}, in the latter -- that $uv$ is of \emph{type II for $\{xy',xy''\}$}.
Similarly, each triangle $uvt \in V(G)$ is is either contained in $E(G_{xy})$, for some $xy \in E(D)$ (so it is of \emph{type I for $xy$}), or $u=v_{x,y}$, $v=v_{x,y'}$ and $t=v_{x,y''}$ for some $x \in U$ (so it is of \emph{type II for $x$}).
We claim the following.

\begin{claim}\label{cla:itt-claw}
$(G,X',Y',E')$ is a yes-instance of \itte if and only if $(D',U,\wei,|\cE|)$ is a yes-instance of \MWM.
\end{claim}
\begin{claimproof}
First, assume that $(D',U,\wei,|\cE|)$ is a yes-instance of \MWM, i.e., there exists a matching $M$ such that $\wei(M)\geq |\cE|$ and $M$ covers $U$.

We note that since $M$ is a matching, in particular, for every $xy \in E(D)$ it holds that $|T(xy) \cap M| \leq 1$.
We set $T:=\bigcup_{xy\in\cE}T(xy)$.
By the definition of $\wei$, we have $\sum_{e \in M} \wei(e)=\sum_{e \in M \cap T} \wei(e)=|M \cap T|$.
Since $ \sum_{e \in M} \wei(e) \geq |\cE|$, we obtain that $|M \cap T| \geq |\cE|$, i.e., $M$ contains at least $|\cE|$ edges from $T$.
However, as $|T(xy) \cap M| \leq 1$ for every $xy \in E(D)$, we know that the following property is satisfied:
\begin{align*}
(*) \textrm{ The matching $M$ contains exactly one edge from each $T(xy)$, where $xy \in \cE$.} 
\end{align*}

Each $x \in U$ is incident to exactly one edge of $M$, denote that edge by $m(x)$.
We define the solution $X$ to $(G,X',Y',E')$ by defining $\eta(xy) \cap X$ for every $xy \in E(D)$.
Let $A(xy)$ be as follows.
\begin{align*}
A(xy):=\begin{cases}
\{v_{x,y},v_{y,x}\} & \textrm{if } m(x)=xy, \\
\{v_{x,y}\} & \textrm{if } m(x)=xt_{xy}, \\
\{v_{y,x}\} & \textrm{if } m(y)=yt_{xy}, \\
\emptyset & \textrm{otherwise.}
\end{cases}
\end{align*}
As $M$ is a matching, we cannot have $m(x)=xy$ and $m(y)=yt_{xy}$, or $m(x)=xt_{xy}$ and $m(y)=yt_{xy}$ at the same time, so $A(xy)$ is well-defined.

Observe that $A(xy) \in \cA(xy)$, by definition of $E(D')$.
Indeed, first three cases in the definition of $A(xy)$ correspond precisely to conditions (1)-(3) in the construction of $E(D')$. 
For the last case, assume that the $\emptyset\notin \cA(xy)$, i.e., $xy \in \cE$. 
However, by $(*)$, if $xy \in \cE$, then $M \cap T(xy) \neq \emptyset$, i.e., one of first three cases in the definition of $A(xy)$ holds.

As $A(xy) \in \cA(xy)$, let $X(xy)$ be a solution to $\cI_{A(xy)}$.
We claim that $X=\bigcup_{xy \in E(D)} X(xy)$ is a solution to the instance $(G,X',Y',E')$ of \itte.
For that, we need to verify that (i) $X$ is an independent set, (ii) $X' \subseteq X$, (iii) $Y' \subseteq Y$, (iv) for every $e\in E'$ it holds that $e \cap X \neq \emptyset$, and (v) $G-X$ is triangle-free.

For (i), assume that $u,v \in X$ are adjacent.
If $u,v$ are of type I for some $xy$, then $u,v \in X(xy)$, so $X(xy)$ is not a solution for $\cI_{A(xy)}$.
Hence, $uv$ is of type II for some $\{xy,xy'\}$.
This means that $u=v_{x,y} \in X(xy), v=v_{x,y'} \in X(xy')$, and then $v_{x,y} \in A(xy)$, and $v_{x,y'} \in A(xy')$.
However, this implies, respectively, that $m(x) \in \{xy,xt_{xy}\}$, and $m(x) \in \{xy',xt_{xy'}\}$, but $\{xy,xt_{xy}\} \cap \{xy',xt_{xy''}\}=\emptyset$, a contradiction.

To see that (ii) holds note that each $u \in X'$ belongs to some $\eta(xy)=V(G_{xy})$, so $u \in X(xy) \subseteq X$. 
Similarly, for (iii), $u \in X \cap Y'$ belongs to some $\eta(xy)=V(G_{xy})$.
Hence, $u \in X(xy)$, but then $X(xy)$ is not a proper solution to $\cI_{A(xy)}$.
Condition (iv) follows directly from the fact that $X(xy)$ is a solution to $\cI_{A(xy)}$ if $uv \in E'$ is of type I for some $xy$.
If $uv \in E'$ is of type II for $\{xy',xy''\}$, recall that by Reduction Rule \ref{red:itt-claw-i}, $|E' \cap E(x)| \leq 1$, so $E' \cap E(x)=\{v_{x,y'}v_{x,y''}\}$.
Hence, by step (4) of the construction of $E(D')$, edges $xy$ and $xt_{xy}$ are not in $E(D')$.
Since $M$ must contain an edge incident to $x$, we have $\{xy',xt_{xy'},xy'',xt_{xy''}\} \cap M \neq \emptyset$, so by definition of $A(xy')$ and $A(xy'')$, we have $\{u,v\} \cap X=\{v_{x,y'},v_{x,y''}\} \cap X \neq \emptyset$.

Finally, to see that (v) holds, consider a triangle $uvw$ in $G$. 
We need to show that $\{u,v,w\} \cap X \neq \emptyset$.
If $uvw$ is of type I for some $xy$, then again, since $X(xy)$ is a solution to $\cI_{A(xy)}$, we know that at least one of $u,v,w$ must belong to $X(xy) \subseteq X$.
So assume that $uvw$ is of type II for some $x \in U$.
As $x \in U$, $M$ must contain an edge incident to $x$, i.e., $\{xy,xt_{xy},xy',xt_{xy'},xy'',xt_{xy''}\} \cap M \neq \emptyset$.
Hence, by definition of sets $A(xy)$, $A(xy')$, and $A(xy'')$, we have $\{u,v,w\} \cap X=\{v_{x,y},v_{x,y'},v_{x,y''}\} \cap X \neq \emptyset$.

\medskip

For the other implication, assume that $(G,X',Y',E')$ is a yes-instance of \itte, so there exists a solution $X$. 
For each $xy \in E(D)$, let $X(xy):=X \cap V(G_{xy})$.
We note that $X(xy)$ is the solution for the instance $\cI_{A(xy)}$, where $A(xy)=X \cap \{v_{x,y},v_{y,x}\}$.

We construct a solution $M$ for the instance $(D',V(D),\wei,|\cE|)$ by defining $M(xy)=M \cap T(xy)$ for each edge $xy \in V(D)$.
\begin{align*}
M(xy):=\begin{cases}
\{xy\} & \textrm{if } A(xy)=\{v_{x,y},v_{y,x}\}, \\
\{xt_{xy}\} & \textrm{if } v_{x,y} \neq v_{y,x}\textrm{ and } A(xy)=\{v_{x,y}\}, \\
\{yt_{xy}\} & \textrm{if } v_{x,y} \neq v_{y,x}\textrm{ and } A(xy)=\{v_{y,x}\}, \\
\emptyset & \textrm{if } A(xy)=\emptyset.
\end{cases}
\end{align*}
We note that $M(xy)$ is well-defined, as the edges $xy, xt_{xy},$ and $yt_{xy}$ were added to $D'$ in steps (1)-(3) of the construction, and if an edge $xy$ ($xt_{xy}$, respectively) is removed in step (4), then $v_{x,y}\notin X$, so $v_{x,y}\notin A(xy)$.
We claim that $M=\bigcup_{xy \in E(D)}M(xy)$ is a matching $M$ in $D'$ such that for each $x \in U$ there exists an edge in $M$ incident to $x$, and $\wei(M)\geq |\cE|$.

First, we prove that (i) $M$ is a matching such that (ii) for each $x \in U$ there exists an edge in $M$ incident to $x$.
For (i), note that the only elements of $V(D')-U$, are vertices of form $t_{xy}$ for some $xy$.
Observe that for a fixed $t_{xy}$ we never have $xt_{xy},yt_{xy} \in M$, by definition of $M(xy)$.
By the same reason, if $x \in U$, we never have $xy,xt_{xy} \in M$.
On the other side, we have $N_{D'}(x) \subseteq \{y,t_{xy},y',t_{xy'},y'',t_{xy''}\}$.
Without loss of generality, assume that $M(xy) \cap \{xy,xt_{xy}\} \neq \emptyset$ and $M(xy') \cap \{xy',xt_{xy'}\} \neq \emptyset$.
Definitions of $M(xy)$ and $M(xy')$ imply that $v_{x,y} \in A(xy)$ and $v_{x,y'} \in A(xy')$.
Since $v_{x,y}$ is adjacent to $v_{x,y'}$, and $A(xy), A(xy') \subseteq X$, this is a contradiction with $X$ being independent.
Hence, at most one of the sets $M(xy) \cap \{xy,xt_{xy}\}$, $M(xy') \cap \{xy',xt_{xy'}\}$ and $M(xy'') \cap \{xy'',xt_{xy''}\}$ is non-empty, i.e., at most one edge of $M$ is incident to $x$.
That proves (i).
For (ii) note that as $E(x)=\{v_{x,y},v_{x,y'},v_{x,y''}\}$ induces a triangle, exactly one of the vertices from $E(x)$ must belong to $X$.
By symmetry, assume that $v_{x,y} \in X$. 
By definition of $M(xy)$, this means that $M \cap \{xy, xt_{xy}\} \neq \emptyset$.
Hence, we get that exactly one edge of $M$ is incident to $x$.

It remains to show that (iii) $\sum_{e \in M} \wei(e) \geq |\cE|$. 
Recall that $\sum_{e \in M} \wei(e) = |T \cap M|$.
It is enough to show that for every $xy \in E(D)$ such that $T(xy) \subseteq T$ (i.e., $\emptyset \notin \cA(xy)$), we have $|T(xy) \cap M|=1$.
Recall that as $D'[T(xy)]$ is a subgraph of a triangle and $M$ is a matching, we clearly have $|T(xy) \cap M|\leq 1$.
On the other hand, since $\emptyset \notin \cA(xy)$, and by Reduction Rule~\ref{red:itt-claw-ii}, $\cA(xy) \neq \emptyset$, we must have $\cA(xy) \cap \{\{v_{x,y},v_{y,x}\},\{v_{x,y}\},\{v_{y,x}\}\} \neq \emptyset$.
By the definition of $M(xy)$, this corresponds precisely to $M \cap T(xy)$ being non-empty.
Hence, $|T(xy) \cap M|=1$, and this concludes the proof of the claim.
\end{claimproof}

Therefore, for an instance $(G,X',Y',|\cE|)$ of \itte, we create an equivalent (by \cref{cla:itt-claw}) instance $(D',U,\wei, |\cE|)$ of \MWM.
Then, we check whether $(D',U,\wei, |\cE|)$ is a yes-instance of \MWM.

It remains to estimate the running time of the algorithm.
We check in polynomial time whether $G$ contains $K_4$.
As for each $xy \in E(D)$, the set $\eta(xy)$ is non-empty, $|E(D)| \leq n$.
For every $xy \in E(D)$ we compute $\cA(xy)$, which requires solving at most four instances $\cI_\emptyset, \cI_{\{v_{x,y}\}},\cI_{\{v_{y,x}\}}, \cI_{\{v_{x,y},v_{y,x}\}}$.
Each of them consists of a graph $G_{xy}$, and by \Cref{thm:strip-structure}, $\tw{G_{xy}} \leq 24$, hence this also can be done in polynomial time.
As the instance $(D',U,\wei, |\cE_N|)$ of \MWM can be constructed and solved in polynomial time by \cref{lem:perfect-matching}. \end{proof}

\newpage
\section{Solving \itte in $S_{2,1,1}$-free graphs}\label{sec:forkfree}
In this section we generalize the algorithm from \cref{thm:itt-claw-free} to the class of $S_{2,1,1}$-free graphs.
The main combinatorial tool is the following theorem used by Lozin and Milani\v{c} to solve \MIS in $S_{2,1,1}$-free graphs~\cite{DBLP:conf/soda/LozinM06}.

\begin{theorem}[Lozin, Milani\v{c}~{\cite[Theorem 4.1]{DBLP:conf/soda/LozinM06}}]\label{thm:prime-graph-claw-free}
Let $G$ be an $S_{2,1,1}$-free prime graph and let $v\in V(G)$. Let $G'$ be an induced prime subgraph of $G-N[v]$. Then $G'$ is claw-free.
\end{theorem}

Equipped with \cref{thm:prime-graph-claw-free}, we can present our algorithm.

\begin{theorem}\label{thm:ittefork}
\itte can be solved in polynomial time in $S_{2,1,1}$-free graphs.
\end{theorem}

\begin{proof}
Let $(G,X',Y',E')$ be an instance of \itte such that $G$ is $S_{2,1,1}$-free.  By calling \cref{lem:X-and-X-star} for the class of $S_{2,1,1}$-free graphs, it is enough to consider $G$ such that $G$ is either a prime $S_{2,1,1}$-free graph or a clique. If $Y'=V(G)$, then we can verify in polynomial time if $X=\emptyset$ is a solution. So since now we assume that $Y'$ is a proper subset of $V(G)$. Observe that if $(G,X',Y',E')$ is a yes-instance, then there exists a solution $X$ that is non-empty. Indeed, for a solution $X=\emptyset$, we can safely add a vertex $v\in V(G)\setminus Y'$.

For every $v \notin Y'$ we define the instance $\cI_v:=(G,X' \cup \{v\},Y',E')$.
We return that $(G,X',Y',E')$ is a yes-instance if and only if there exists $v \notin Y'$ for which $\cI_v$ is a yes-instance.
Hence, let $v \notin Y'$ be fixed.
Clearly, for any solution $X$, we have that $X\cap N(v)=\emptyset$, so if there is a vertex in $N(v)\cap X'$ or there is an edge in $E'$ with both endpoints in $N(v)$, we report a no-instance.

Define $G_v:=G-N[v]$. We initialize $X'':=X'$ and $Y'':=Y'-N[v]$, and let $E''$ contain these edges from $E'$ that have both endpoints in $V(G_v)$. Consider a triangle $xyz$.
If $x\in N(v)$ and $y,z \in V(G_v)$ we add $yz$ to $E''$. If $x,y \in N(v)$ and $z \in V(G_v)$ we  add $z$ to $X''$. Finally, for every edge $uw\in E'$ with $u\in N(v)$ and $w\in V(G_v)$, we add $w$ to $X''$. Hence, it is enough to focus on the instance $(G_v,X'',Y'',E'')$, as it is clearly equivalent to the instance $\cI_v$. Note that this reduction can be performed in polynomial time.

We call again \cref{lem:X-and-X-star}, now for the class $\cY:=\{G_w\ | \ G\in \cX, w\in V(G) \}$, so it is enough to solve the problem for the class $\cY^*$ of all induced subgraphs of the graphs in $\cY$ that are either prime or cliques. Clearly, every clique is claw-free. Together with \cref{thm:prime-graph-claw-free}, it implies that every graph in $\cY^*$ is claw-free. By \cref{thm:itt-claw-free},  \itte  can be solved in polynomial time on $\cY^*$, and therefore it can be solved in polynomial time on $\cY$. Since $G_v\in \cY$, the instance $(G_v,X'',Y'',E'')$ can be solved in polynomial time. 
\end{proof}

Combining \cref{thm:ittereduce} with \cref{thm:ittefork} we immediatey obtain our main algorithmic result.

\mainthm*

\newpage
\section{Minimal obstructions to \coloring{W_k} in claw-free graphs}\label{sec:obstructions}
Recall that for graphs $F,H$, we say that $G$ is a \emph{$F$-free minimal $H$-obstruction} if:
\begin{itemize}
\item $G$ is $F$-free,
\item $G$ does not admit a homomorphism to $H$,
\item every induced subgraph of $G$ admits a homomorphism to $H$.
\end{itemize}
In this section we show the following result.

\begin{theorem}\label{thm:obstructions}
For all odd $k \geq 5$, there are infinitely many claw-free minimal $W_k$-obstructions.
\end{theorem}

Before we proceed to the proof, let us introduce an important gadget, that will be used in the proof of \cref{thm:obstructions}, as well as later in \cref{sec:otherf}.
Consider the graph called a \emph{diamond}, i.e., $K_4$ with one edge removed, see \cref{fig:chain}.
For a positive integer $\ell$, we an \emph{$\ell$-chain of diamonds} is the graph obtained from $\ell$ copies of the diamond by identifying their vertices of degree 2 so that the copies are arranged in a path-like manner, see \cref{fig:chain} (right).
An $\ell$-chain of diamonds has exactly two vertices of the degree 2, we call them \emph{endpoints}.

\begin{figure}[htb]
\centering
\raisebox{0.8\height}{\includegraphics[page=7]{figures}} \hskip 2cm
\includegraphics[page=1]{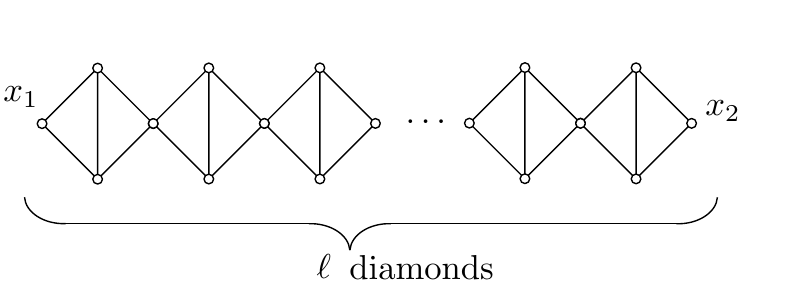}
\caption{The diamond (left), and an $\ell$-chain of diamonds with endpoints $x_1$ and $x_2$ (right).} \label{fig:chain}
\end{figure}

Recall that by $0$ we denote the universal vertex of $W_k$.
The following observation can be easily verified.
\begin{observation}\label{obs:chain-endpoints}
Let $\ell >0$, let $k$ be an odd integer, and let $R$ be the $\ell$-chain of diamonds with endpoints $x_1$ and $x_2$.
Let $\varphi: R \to W_k$. 
Then $\varphi(x_1)=0$ if and only if $\varphi(x_2)=0$.

\noindent On the other hand, if $\ell > k/2$, then (i) for every $i,j \in [k]$ there exists a homomorphism $\varphi_{ij}:R \to W_k$ such that $\varphi(x_1)=i$ and $\varphi(x_2)=j$, and (ii)
there exists a homomorphism $\varphi_0:R \to W_k$ such that $\varphi(x_1)=\varphi(x_2)=0$.
\end{observation}

Let us discuss the intuition behind the construction.
For a 3-regular graph $Q$ and an integer $\ell$, let $Q_\ell$ denote the graph obtained from $Q$ as follows.
For every vertex of $Q$ we introduce a triangle called \emph{vertex triangle}. Each edge of $Q$ is substituted with a $\ell$-chain of diamonds. We do it in such a way that each vertex from a vertex-triangle is connected to exactly one chain of diamonds; this is possible as $G$ is 3-regular.
Note that the graphs constructed in such a way are claw-free and $K_4$-free.

Consider a homomorphism $\phi$ from such constructed graph $Q_\ell$ to $W_k$.
Observe that exactly one vertex of each vertex triangle is mapped by $\phi$ to 0, i.e., the central vertex of $W_k$.
Moreover, for every $\ell$-chain of diamonds, by \cref{obs:chain-endpoints}, either both endpoints are mapped to $0$, or none of them.
Consequently, $\phi$ defines a perfect matching of $Q$, where chains of diamonds whose endpoints are mapped to 0 correspond to the edges of this matching.

Thus if $Q$ \emph{does not} have a perfect matching, then $Q_\ell$ does not admit a homomorphism to $W_k$.
Furthermore, by making $\ell$ larger and larger, we can ensure that the constructed graphs have longer and longer cycles,
so one cannot be a subgraph of another. However, we point out that the minimality condition is not necessarily preserved.
We deal with this in the following lemma.

\begin{lemma}\label{lem:obsconstr}
Let $k \geq 5$ be an odd integer.
For every integer $\ell \geq 4$ there exists a claw-free minimal $W_k$-obstruction $Z_\ell$, such that
\begin{itemize}
\item $Z_\ell$ has no induced cycles of length in $[4,\ell]$,
\item $Z_\ell$ has an induced cycle longer than $\ell$.
\end{itemize}
\end{lemma}
\begin{proof}
Let $Q$ be a 3-regular graph with no perfect matching, e.g., the graph depicted in \cref{fig:nopm}.
Consider the graph $Q_\ell$ and recall that it is claw-free and $K_4$-free, and does not have a homomorphism to $W_k$.
Also note that every induced cycle in $Q_\ell$ is either a triangle or is longer than $\ell$. 

\begin{figure}[htb]
\centering
\includegraphics[page=3]{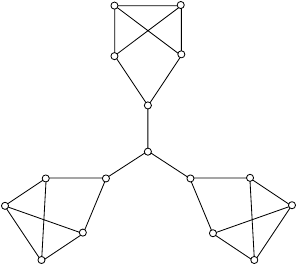}
\caption{An example of a 3-regular graph with no perfect matching.}\label{fig:nopm}
\end{figure}

Let $Z_\ell$ be a minimal induced subgraph of $Q_{\ell}$ which does not have a homomorphism to $W_k$.
We can obtain $Z_\ell$ from $Q_\ell$ by deleting vertices one by one until we obtain a graph with desired property.
Clearly $Z_\ell$ has no induced cycles of length in $[4,\ell]$ and is $K_4$-free. 

Now suppose that is does not have an induced cycle longer than $\ell$, i.e., that all induced cycles in $Z_\ell$ are triangles.
This means that $Z_\ell$ is chordal and thus perfect.
As $Z_\ell$ does not contain $K_4$, we conclude that it is 3-colorable, so in particular admits a homomorphism to $W_k$. This contradicts the definition of $Z_\ell$.
\end{proof}

Now we are ready to prove \cref{thm:obstructions}.
\begin{proof}[Proof of \cref{thm:obstructions}]
Let $G_1$ be the graph given by \cref{lem:obsconstr} for $\ell = 4$, i.e., $G_1 = Z_4$.
From now on assume that we have already constructed pairwise distinct graphs $G_1,\ldots,G_i$, each of which is a claw-free minimal $W_k$-obstruction. 

Now let $G_{i+1}$ be the graph obtained by applying \cref{lem:obsconstr} for $\ell$ equal to the length of a largest induced cycle in a graph in $\{G_1,\ldots,G_i\}$. Clearly $G_{i+1}=Z_\ell$ is a claw-free minimal $W_k$-obstruction.
We only need to make sure that it is distinct than any graph constructed so far.
However, no graph in $G_1,\ldots,G_i$ has an induced cycle with more than $\ell$ vertices,
while $G_{i+1}$ has such a cycle.
Consequently, iterating this process we can construct a family of arbitrary size.
\end{proof}

\newpage
\section{\colorext{W_5} in $S_{3,3,3}$-free graphs is hard}\label{sec:hardness}
In this section we prove the following hardness result.

\begin{theorem}
The \colorext{W_5} problem is \NP-hard in $S_{3,3,3}$-free graphs of maximum degree at most $5$.
Furthermore, it cannot be solved in subexponential time, unless the ETH fails.
\end{theorem}

\begin{proof}
We  reduce from \coloring{3} in claw-free graphs of maximum degree at most $4$ which is known to be \NP-hard and cannot be solved in subexponential time, unless the ETH fails~\cite{Holyer1981TheNO}.
Let $G$ be an instance of \coloring{3} such that $G$ is a claw-free graph of maximum degree at most $4$. 
Note that we can assume that the minimum degree of $G$ is at least $3$.
Indeed, if there is a vertex $v$ of $G$ of degree at most $2$, and we can properly color $G-v$, then there will be always a color left for $v$.
Therefore, we can exhaustively remove all vertices of degree at most $2$, obtaining an equivalent instance.
Note that if $G$ is claw-free and on minimum degree 3, every vertex of $G$ belongs to a triangle.

We construct an instance $(\tG,U,\vphi)$ of \colorext{W_5} as follows. 
We initialize $V(\tG):=V(G)$, $E(\tG):=E(G)$, and $U:=\emptyset$. For each $v\in V(G)$ we proceed as follows. 
We add to $V(\tG)$ vertices $x_v,y_v,z_v$ and to $E(\tG)$ we add $x_vy_v$, $y_vz_v$, and $y_vv$. 
Moreover we add $x_v,z_v$ to $U$ and set $\vphi(x_v):=0$ and $\vphi(z_v):=4$. 
This completes the construction of the instance $(\tG,U,\vphi)$.

By the construction and the fact that $G$ has maximum degree at most $4$, we observe that $\tG$ has maximum degree at most $5$. Let us show that $\tG$ is $S_{3,3,3}$-free. 
Suppose to the contrary that there is an induced $S_{3,3,3}$ in $\tG$. 
Denote its vertices as $s_0$ and $s_{i,j}$ for $i,j\in [3]$, so that $s_0$ is the vertex of degree $3$ in $S_{3,3,3}$ and for $i\in [3]$ vertices $s_{i,1},s_{i,2},s_{i,3}$ are consecutive vertices of a path and $s_{i,1}$ is adjacent to $s_0$.
Suppose first that $s_0=v\in V(G)$. Since $G$ is claw-free one of $s_{1,1},s_{2,1},s_{3,1}$, say $s_{1,1}$, must be $y_v$.
Then the path on vertices $s_{1,1},s_{1,2},s_{1,3}$ must be contained in $\{x_v,y_v,z_v\}$ with $y_v$ being an endvertex, but there is no such path, a contradiction.
So now suppose that $s_0\notin V(G)$. Since $s_0$ is of degree $3$ in $S_{3,3,3}$, the vertex $s_0=y_v$ for some $v\in V(G)$.
Then two of $s_{1,1},s_{2,1},s_{3,1}$, say $s_{1,1},s_{2,1}$, must be in $\{x_v,z_v\}$.
Since $x_v,z_v$ are of degree $1$, there is no vertex in $\tG$ that can be $s_{1,2}$ or $s_{2,2}$, a contradiction.
Hence, $\tG$ is $S_{3,3,3}$-free.

So let us verify the equivalence. Suppose first that there exists a $3$-coloring $f$ of $G$, and assume that this coloring uses colors $\{0,1,2\}$. Define $\psi: \tG \to W_5$ as follows.
For every $v\in V(G)$, define $\psi(v):=f(v)$ and set $\psi(y_v)$ as a common neighbor of the vertices $0,f(v),4$ in $W_5$ -- note that for every $f(v)\in\{0,1,2\}$ such a common neighbor exists. 
For every $u\in U$, we set $\psi(u):=\vphi(u)$.
This completes the definition of $\psi$.
Let us verify that $\psi$ is a homomorphism from $\tG$ to $W_5$ that extends $\vphi$.
The latter follows immediately from the definition of $\psi$.
Let $uw\in E(\tG)$. If both $u,w\in V(G)$, then $\psi(u),\psi(w)$ are two distinct vertices of $\{0,1,2\}$ and hence $\psi(u)\psi(w)\in E(W_5)$. So assume that $u\notin V(G)$. Then $uw\in \{x_vy_v,y_vz_v,y_vv\}$ for some $v\in V(G)$. It can be verified that in every case it holds that $\psi(u)\psi(w)\in E(W_5)$.

So now suppose that there exists a homomorphism $\psi: \tG \to W_5$ that extends $\vphi$. 
Let us show first that for every vertex $v\in V(G)$ it holds $\psi(v)\in \{0,1,2\}$. 
Let $v\in V(G)$. Since $\psi$ extends $\vphi$ it holds that $\psi(x_v)=0$ and $\psi(z_v)=4$. 
Therefore $\psi(y_v)\in \{3,5\}$ and thus $\psi(v)\in \{0,1,2,4\}$.
Suppose that there is a vertex $v\in V(G)$ such that $\psi(v)=4$.
Recall that $v$ is contained in some triangle $vuw$ in $G$.
So $\psi(v)\psi(u)\psi(w)$ must be a triangle in $W_5[\{0,1,2,4\}]$ containing $4$, but there is no such a triangle, a contradiction.
Hence, for every $v\in V(G)$ we have $\psi(v)\in \{0,1,2\}$.
Define a $3$-coloring $f: V(G) \to \{0,1,2\}$ so that $f(v):=\psi(v)$ for every $v\in V(G)$.
Since $\psi$ is a homomorphism, for every edge $uv\in E(G)$, it holds that $f(u)\neq f(v)$.

Finally, note that the number of vertices of $\tG$ is linear in the number of vertices of $G$, so the ETH lower bound follows.
This completes the proof.
\end{proof}

\newpage
\section{\coloring{W_k} in $F$-free graphs: beyond paths and subdivided claws}\label{sec:otherf}
\label{sec:otherf}
Recall that for an odd integer $k \geq 5$, by $W_k$ we denote the graph obtained from the $k$-cycle with vertices $1,2,\ldots,k$ by adding a universal vertex $0$.
Let $F$ be a connected graph.
We aim to show that if $F$ is neither a path nor a subdivided claw,
then for all $k \geq 5$ the problem is \NP-hard. We do it in several steps.

\begin{proposition}\label{prop:forbiddencycles}
Let $k \geq 5$ be an odd integer.
For every $g$, the \coloring{W_k} problem is \NP-hard in graphs of girth at least $g$.
Furthermore, the problem cannot be solved in subexponential time, unless the ETH fails.
\end{proposition}
\begin{proof}
We reduce from the \coloring{C_k} problem in graphs with maximum degree 3~\cite{DBLP:conf/esa/ChudnovskyHRSZ19}, let $G$ be an instance.
We construct an instance $G'$ of \coloring{W_k} as follows.
For each vertex $v$ we will introduce a vertex gadget $Y(v)$, which is a graph with a designated three-element subset $I(v) \subseteq V(Y(v))$, which satisfies the following properties:
\begin{itemize}
\item $Y(v)$ admits a homomorphism to $W_k$,
\item in every homomorphism from $Y(v)$ to $W_k$, all vertices in $I(v)$ are mapped to the same element of $[k]$,
\item the girth of $Y(v)$ is at least $g$,
\item elements of $I(v)$ are at distance at least $g$ in $Y(v)$,
\item the size of $Y(v)$ depends only on $k$ and $g$.
\end{itemize}
For a moment let us assume that we can construct such gadgets.
For each edge $vw$ of $G$, we introduce an edge joining one element of $I(v)$ and one element in $I(w)$.
We do it in a way that each element from each $I(v)$ has at most one neighbor outside $Y(v)$, we can do it as $G$ is of maximum degree 3.

It is straighforward to observe that the properties of $Y(v)$ imply that $G'$ is a yes-instance of \coloring{W_k} if and only if $G$ is a yes-instance of \coloring{C_k}. Furthermore the constructed graph $G'$ has girth at least $g$.
Finally, as the number of vertices of $G'$ is linear in the number of vertices of $G$, the ETH lower bound also follows.

Thus the only thing left is to construct $Y(v)$. 
By the result of Zhu~\cite{DBLP:journals/jgt/Zhu96}, for every \emph{core} $H$ and every $\ell \geq 1$ there is a graph $Z$ with girth at least $\ell$
which is uniquely $H$-colorable, i.e., there exists a homomorphism $\varphi$ from $Z$ to $H$ with the property
that every homomorphism from $Z$ to $H$ can be obtained from $\varphi$ by combining it with an automorphism of $H$.
Let us skip the definition of a core, as it is not relevant for this paper; it is sufficient to know that $W_k$ is a core for every odd $k \geq 5$~\cite[Section 1.4]{hell2004graphs}.

Let $Z$ we obtained by calling this theorem for $H=W_k$ and $\ell := \max( g, 2(k+1) )$.
Let $\phi : Z \to W_k$ be the unique (up to automorphisms of $W_k$) homomorphism from $Z$ to $W_k$.
The graph $Z$ contains a cycle with at least $2(k+1)$ vertices, so there are at least $k+1$ vertices $x$
such that $\phi(x) \neq 0$. Consequently, there are two vertices $x,x'$ of $Z$ with $\phi(x)=\phi(x') \in [k]$.

Now let $Y(v)$ be the graph defined as follows.
We start with the subdivided claw $S_{\ell,\ell,\ell}$.
Then each edge $ab$ is substituted with a copy of $Z$, with $a$ identified with $x$ and $b$ identified with $x'$.
The copies of $x,x'$ that were not identified, i.e., the ones corresponding to the leaves 
of the subdivided star, form the set $I(v)$ (see \cref{fig:sausage}).

\begin{figure}[htb]
\centering
\raisebox{2.4\height}{\includegraphics[page=11]{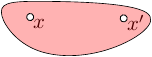}} \hskip 2cm
\includegraphics[page=10]{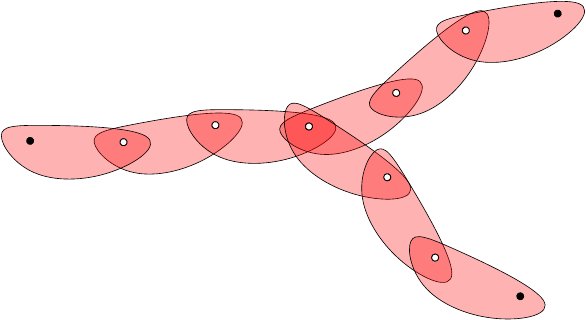}
\caption{A schematic view of the graph $Z$ (left), and the graph $Y(v)$ constructed for $\ell=3$ (right). Black vertices form the set $I(v)$.} \label{fig:sausage}
\end{figure}

The properties of $Z$ and the automorphisms of $W_k$ imply that $Y(v)$ and $I(v)$ satisfy all desired properties.
Finally, as $G$ is of maximum degree 3 and $Z$ depends only on $H$ and $g$, i.e., if of constant size,
we conclude that the number of vertices in the constructed graph is linear in $|V(G)|$.
This completes the proof.
\end{proof}

For an integer $g$, let $\cX_g$ denote the class of graphs that satisfy the following properties:
\begin{itemize}
\item each $G \in \cX_g$ is of maximum degree 4,
\item each $G \in \cX_g$ is $K_{1,4}$-free,
\item for each $G \in \cX_g$ and any two vertices $u,v$ of $G$, each of which has three independent neighbors,
it holds that either $u$ and $v$ are at distance at least $g$, or $N[u]=N[v]$.
\end{itemize}

To show hardness for \coloring{W_k} in $\cX_g$, we will show a reduction from the \textsc{Pos-1-in-3-Sat} problem,
whose instance $(X,\cC)$ consists of a set $X$ of variables, and a set $\cC$ of clauses, such that each clause has exactly three elements, and all literals in the clauses are positive.
We ask for an assignment $\sigma:X \to \{0,1\}$ such that exactly one literal from each clause satisfied, i.e., the variable corresponding to that literal is mapped by $\sigma$ to $1$.

The following result is folklore; for completeness we include its proof in the appendix.

\begin{restatable}{theorem}{oneinthreesat}
\label{thm:oneinthreesat}
The \textsc{Pos-1-in-3-Sat} problem in \NP-hard, even if each variable occurs in at most 6 clauses.
Furthermore, it cannot be solved in subexponential time, unless the ETH fails.
\end{restatable}

Now we are ready to prove the following.

\begin{proposition}\label{prop:forbiddenforests}
Let $k \geq 5$ be an odd integer. 
For each $g$, the \coloring{W_k} problem is \NP-hard in graphs in $\cX_g$.
Furthermore, the problem cannot be solved in subexponential time, unless the ETH fails.
\end{proposition}
\begin{proof}
Let $\ell=\max\{g,(k+1)/2\}$.
We reduce from \textsc{Pos-1-in-3-SAT}, such that each variable occurs in clauses at most 6 times, see \cref{thm:oneinthreesat}.
Consider an instance with variables $X=\{x_1,\ldots,x_n\}$ and clauses $\cC=\{C_1,\ldots,C_m\}$.

For each $x \in X$ by $p_x$ we denote the number of all occurrences of $x$ in clauses.
Also, for each $x \in X$ we define an arbitrary order of the occurrences of the variable $x$.

We need to construct an instance $G$ of \coloring{W_k}, such that $G \to W_k$ if and only if there exists an assignment $\sigma:X \to \{0,1\}$ such that exactly one literal from each clause is mapped to $1$.
 
We introduce two types of gadgets, variable gadget, and clause gadget.
A variable gadget for the variable $x$ is a pair $(F(x),W(x))$, such that $F(x)$ is a graph, and $W(x) = \{w_1(x),\ldots,w_{k_x}(x)\} \subseteq V(F)$. 
Graph $F$ consists of $p_x$ copies of a crown (see \cref{fig:var-gadget} (left)), and $p_x-1$ copies of the $\ell$-chain of diamonds, with endpoints identified with the vertices of the crowns as shown on \cref{fig:var-gadget}.
Observe that vertices of $F(x)$ can be partitioned into three independent sets $F_0,F_1,F_2$.
Vertex $w_i(x) \in W$ corresponds to $i$-th occurrence of $x$ in clauses.
We observe that and for any homomorphism $\varphi:F \to W_k$, and for each $v_i,v_j \in F_1$, we have $\varphi(v_i)=0$ if and only if $\varphi(v_j)=0$, i.e., either all vertices from $F_0$ are mapped to $0$, or none of them.
In particular, as $W(x) \subseteq F_0$, either all vertices from $W(x)$ are mapped to $0$, or none of them.
Intuitively, mapping vertices form $W(x)$ to $0$ will correspond to setting the corresponding variable $x$ true (i.e.,to $1$).

\begin{figure}[htb]
\centering
\includegraphics[page=6, scale=0.85]{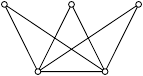} \hskip 1cm
\includegraphics[page=8]{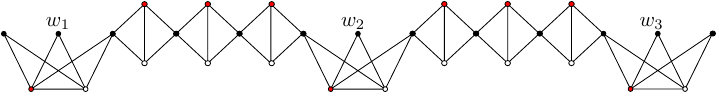}
\caption{A crown (left) and a variable gadget for $x \in X$ (right). Here, $p_x=4$, and $\ell=3$.
We depict vertices from $F_0,F_1,F_2$, respectively, by black, red and white.
Observe that in a homomorphism form $F$ to $W_k$ either all vertices from $F_0$ are mapped to $0$, or none of them.} \label{fig:var-gadget}
\end{figure}

For a clause $C$, the clause gadget $K(C)$ is a copy of $K_3.$
Each vertex of $K(C)$ corresponds to one literal of the clause $C$.
Recall that in every homomorphism $h:K_3 \to W_k$, exactly one vertex of $K_3$ must be mapped to $0$.

We construct $G$ as follows. 
For every variable we introduce a variable gadget, and for every clause we introduce a clause gadget. 
Let $C=(u_1,u_2,u_3)$ be a clause, and let $c_1,c_2,c_3$ be the vertices of $K(C)$, corresponding, respectively, to the literals $u_1,u_2,u_3$.
Assume $u_1$ ($u_2,u_3$, respectively) is an $i$-th ($j$-th, $p$-th, respectively) occurrence of some variable $x$ ($y,z$, respectively).
We introduce $g$-chain of diamonds $Q$, and identify one endpoint of $Q$ with $u_1$, and another with $w_i(x)$ ($u_2$ and $w_j(y)$, and $u_3$ and $w_p(z)$, respectively), see \cref{fig:cla-gadget}.
That concludes the construction of $G$.

\begin{figure}[htb]
\centering
\includegraphics[page=9]{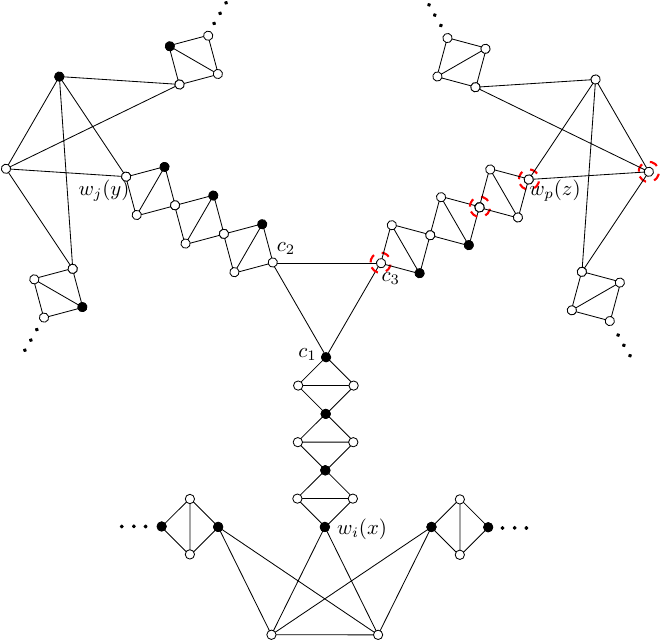}
\caption{Connecting gadgets of variables $x,y,$ and $z$ with a clause gadget $C=(u_1,u_2,u_3) \in \cC$ for $\ell=3$.
Black vertices are vertices mapped to $0$ if $x$ is true, and $y,z$ are false, 
Also, four vertices marked red correspond to four types of degree-4 vertices in $G$.} \label{fig:cla-gadget}
\end{figure}

Given a homomorphism $\varphi:G \to W_k$, we define the assignment $\sigma:X \to \{0,1\}$ such that $\sigma(x)=1$ if and only if the vertices from set $W(x)$ are mapped to zero.
Note that if $C=(u_1,u_2,u_3)$ is a clause gadget, exactly one of the vertices of $K(C)$, say $c_1$ is mapped to $0$.
That corresponds to setting the variable $u_1$ true, and the remaining variables that occur in $C$ false.
It is straightforward to verify that this way exactly one literal from each clause is mapped to true by $\sigma$.

Conversely, given the assignment $\sigma:X \to \{0,1\}$ vertices from $W(x)$, we define $\varphi: V(G) \to W_k$ as follows.
Consider the variable gadget $(F(x),W(x))$.
If $\sigma(x)=1$, we map vertices of $F_0$, $F_1$, $F_2$, respectively to $0$, $1$ and $2$.
Otherwise, (i.e., if $\sigma(x)=0)$, we map them respectively to $1$, $0$ and $2$.
Now let $K_3(C)$ be a clause gadget, with vertices $c_1,c_2,c_3$ corresponding, respectively, to variables $x,y,z$.
By the properties of $\sigma$, exactly one of the variables, say $x$, must be mapped to true, while $y$ and $z$ are mapped to false.
We set $\varphi(c_1)=0$, $\varphi(c_2)=1$ and $\varphi(c_3)=2$.

Finally, let $R$ be an $\ell$-chain of diamonds that connects variable gadget $(F(x),W(x))$ with some clause gadget $K(C)$.
Observe that if $\sigma(x) =1$, then endpoints of $R$ are mapped to $0$.
Otherwise, endpoints of $R$ are mapped to elements of $\{1,2\}$. 
By \cref{obs:chain-endpoints}, we can extend $\varphi$ to $V(R)$ in a way that $\varphi$ is still a homomorphism to $W_k$.

We show that $G$ belongs to $\cX_g$. 
Clearly, every vertex of $G$ has at most four neighbors.
To see that $G$ is $K_{1,4}$-free we note that each every vertex $v$ of $G$ that has four neighbors either belongs to some chain of diamonds, or is a vertex from the crown in some variable gadget (see \cref{fig:cla-gadget}).
In both cases $v$ belongs to some triangle, so $G$ does not contain $K_{1,4}$.
Finally, note that if $u$ and $v$ are vertices such that each of them has three independent neighbors, $u$ and $v$ must belong to the copies of crowns in the variable gadgets. 
If $u$ and $v$ belong to the same crown, then, by our construction, $N[u]=N[v]$.
Otherwise, the distance between $u$ and $v$ is at least $2\ell+1 \geq g$, since every path between two different copies of crown must use $2\ell+1$ vertices of some $\ell$-chain of diamonds.

It remains to note that $G$ has $O(\ell(n+m))$ vertices -- at most $(30+15\ell)n$ coming from variable gadgets, and at most $9\ell m$ coming from the clause gadgets and $\ell$-chain of diamonds linking the gadgets.
As $\ell$ is a constant, this gives the desired bound and concludes the proof.
\end{proof}

Let us point out that the degree bound in \cref{prop:forbiddenforests} cannot be improved to $3$.
Indeed, by Brooks' theorem, every connected graph $G$ of maximum degree 3 which is not a $K_4$ can be properly colored with 3 colors. Clearly every 3-colorable graph admits a homomorphism to $W_k$, for every $k \geq 5$.
On the other hand, as $W_k$ is $K_4$-free, $K_4$ does not admit a homomorphism to $W_k$.

Let us summarize the results of this section.
\thmhardness*
\begin{proof}
Let $k \geq 5$ be an odd integer and let $F$ be connected.
If $F$ contains a cycle, then we obtain hardness by applying \cref{prop:forbiddencycles} for $g= |V(F)|+1$;
clearly a graph with girth at least $|V(F)|+1$ cannot contain $F$ as an induced subgraph.

So let us assume that $F$ is a tree which is not path nor a subdivided claw.
This means that $F$ has (i) a vertex of degree at least 4 or (ii) two vertices of degree 3.
In this case we get hardness by applying \cref{prop:forbiddenforests} for $g = |V(F)|+1$.
Indeed, graphs in $\cX_g$ are $K_{1,4}$-free by definition.
So assume that $F$ satisfies (ii), and let $a,b \in V(F)$ be of degree 3.
Clearly, both $a$ and $b$ have three independent neighbors, and $N[a] \neq N[b]$.
However, if $G \in \cX_g$, and $u,v \in V(G)$ are some vertices, each with three independent neighbors, either $N[u]=N[v]$, or the distance between $u$ and $v$ is at least $g$,
thus $F$ is certainly not an induced subgraph of $G$.

The only case left if that $k \geq 7$ and $F$ is a subdivided claw.
However, in this case Feder and Hell~\cite{edgelists} proved hardness already for line graphs,
which are in particular claw-free. Their reduction implies the ETH lower bound.
This completes the proof.
\end{proof}

\newpage
\section{Conclusion}\label{sec:conclusion}
Let us conclude the paper with pointing out some open questions and directions for future research.

\paragraph{Variants of \coloring{W_5} in $S_{a,b,c}$-free graphs.}
Recall that \colorext{W_5} is polynomial-time solvable in $S_{2,1,1,}$-free graphs but \NP-hard in $S_{3,3,3}$-free graphs.
We point out that this leaves an infinite family of open cases, and the minimal ones are $S_{2,2,1}$ and $S_{3,3,3}$.

Initial research shows that \cref{obs:pre-ext-triangle-free} can probably be extended to $S_{2,2,1}$-free graphs,
at least without precolored vertices. Thus there is hope to solve \coloring{W_5} by a reduction to \itte.
However, an analog of \cref{obs:pre-ext-triangle-free} does not hold for $S_{2,2,2}$-free and for $S_{3,1,1,}$-free graphs; see \cref{fig:trianglefree}.

\begin{figure}[htb]
\centering
\includegraphics[page=5]{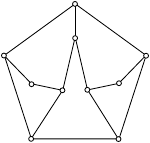} \hskip 2cm
\includegraphics[page=4]{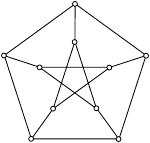} 
\caption{An $\{S_{2,2,2},K_3\}$-free (left) and an $\{S_{3,1,1},K_3\}$-free (right) graph that are not $C_5$-colorable.} \label{fig:trianglefree}
\end{figure}

Of course this does not mean that some other approach cannot work for \coloring{W_5}.
This leads to a natural question:
\begin{question}
For which $a,b,c$ are \coloring{W_5} and \colorext{W_5} polynomial-time solvable in $S_{a,b,c}$-free graphs?
\end{question}
In particular, it is interesting to investigate whether \coloring{W_5} is \NP-hard in $S_{a,b,c}$-free graphs for any constant $a,b,c$.

\paragraph{Minimal obstructions.}
Recall that by \cref{thm:obstructions} and \cref{thm:hardness},
the only connected graphs $F$ for which we can hope for a finite family of $F$-free minimal $W_k$-obstructions are paths. 
This leads to the following question.

\begin{question}\label{que:obstructions}
For which $k$ and $t$ is the family of $P_t$-free minimal $W_k$-obstructions finite?
\end{question}

Let us point out that for all $k$, the number of $P_4$-free minimal $W_k$-obstructions is finite.
Indeed, $P_4$-free graphs are perfect. Thus if $G$ is $P_4$-free and does not contain $K_4$, 
then $G$ is 3-colorable (and thus it admits a homomorphism to $W_k$).
On the other hand, $K_4$ is a minimal graph that does not admit a homomorphism to $W_k$.
Concluding, $K_4$ is the only $P_4$-free (even $P_3$-free) $W_k$-obstruction.

On the negative side, recall that there exists an infinite family of $P_7$-free 4-vertex-critical graphs~\cite{DBLP:journals/jct/ChudnovskyGSZ20}.
An inspection of this family shows that these graphs are are minimal $W_k$-obstructions for every odd $k \geq 5$.
Note that for each such graph $G$, the fact that an induced subgraph $G'$ of $G$ has a proper 3-coloring implies that $G'$ has a homomorphism to $W_k$. However, it still needed to be verified that $G$ itself does not a homomorphism to $W_k$.

Thus the open cases in \cref{que:obstructions} are $t=5$ and $t=6$.

%
%
%

\bibliographystyle{abbrv}
\bibliography{main}

\newpage
\section*{Appendix. Solving the \textsc{MWM*} problem}\label{sec:appendixMWM}
\setcounter{theorem}{4}
\perfectmatching*
\begin{proof}
Let $(G,U,\wei,k)$ be an instance of \MWM.
The proof consists of two steps: first, we construct a triple $(G',\wei',2k)$, where $G'$ is a graph, and $\wei':E(G') \to \mathbb{N}$, such that $(G,U,\wei,k)$ is a yes-instance of \MWM if and only if there exists a perfect matching $M'$ in $G'$ such that $\wei'(M')\geq 2k$.
Then, for every such triple $(G',\wei',2k)$, we will construct an instance $(G',\wei'',k')$ of \textsc{Maximum Weight Matching} such that there exists a perfect matching $M'$ in $G'$ and $\wei'(M')\geq 2k$ if and only if $(G',\wei'',k')$ is a yes-instance of \textsc{Maximum Weight Matching}.
In other words, we reduce the \MWM problem to the \textsc{Maximum Weight Matching} problem (in two steps) -- to solve an instance $(G,U,\wei,k)$ of \MWM, we construct an instance $(G',\wei'',k')$ of \textsc{Maximum Weight Matching} as described. 

\smallskip
\noindent\textbf{Step 1:} To obtain $G'$, take two disjoint copies $G_0$ and $G_1$ of $G$.
Denote by $u_0$ and $u_1$, respectively, the vertices in $G_0$ and $G_1$ that correspond to $u \in V(G)$. 
We add to $G'$ set $U^*=\{u_0u_1: u \in U\}$. 
That concludes the construction of $G'$.
We set $\wei'(e)=0$ for every $e \in U^*$, and for every $uw \in E(G)$, we set $\wei'(u_0w_0)=\wei'(u_1w_1)=\wei(uw)$.

\begin{claim}\label{cla:matching-one}
$(G,U,\wei,k)$ is a yes-instance of \MWM if and only if there exists a perfect matching $M'$ in $G'$ such that $\wei'(M')\geq 2k$.
\end{claim}
\begin{claimproof}
Assume that $(G,U,\wei,k)$ is a yes-instance of \MWM, and let $M$ be a matching in $G$ such that $\wei(M)\geq k$, and $M$ covers $U$.
We define $M'_1=\{u_0w_0,u_1w_1: uw \in M\} \subseteq E(G')$ and $M'_2=\{u_0u_1: M$ does not cover $u\} \subseteq E(G')$, and let $M' =M'_1 \cup M'_2$.
To see that $M'$ is a perfect matching, consider a vertex $u_i \in V(G')$ for some $i \in \{0,1\}$.
Assume that there exists two neighbors $x,y$ of $u_i$ such that $xu_i, yu_i \in M'$.
Clearly, as $u_i$ can have at most one neighbor outside $G_i$, we can assume that $y=w_i$ for some $w_i \in V(G_i)$, so by definition of $M'$, we have $uw \in M$, and in particular, $M$ covers $u$. 
Therefore, if $x=v_i$ for some $v_i \in V(G_i)$, we also have that $uv \in M$, a contradiction with $M$ being a matching.
Otherwise, $x=u_{1-i}$, but the edge $u_iu_{1-i}$ belongs to $M'$ if and only if $M$ does not cover $u$, and that also gives a contradiction.
Hence, $M'$ is a matching. 
To see that $M'$ is perfect, note that if $u \in V(G)$ is incident to an edge $uw \in M$, then $u_iw_i \in M'_1$. 
Otherwise, $M$ does not cover $u$, and then $u_iu_{1-i} \in M'_2$.
Finally, observe that $\sum_{e \in M'} \wei'(e) \geq 2 \sum_{e \in M} \wei(e) \geq 2k$.

Now assume that there exists a perfect matching $M'$ in $G'$ such that $\wei'(M')\geq 2k$.
As for every $e \in U^*$ it holds that $\wei'(e)=0$, we must have $\sum_{e \in M' \cap E(G_1)} \wei'(e) \geq k$ or $\sum_{e \in M' \cap E(G_2)} \wei'(e) \geq k$.
Assume by symmetry that the first case holds. 
We define a subset $M$ of $E(G)$ as follows: $uv \in M$ if and only if $u_1v_1 \in M'$.
Clearly, $M$ is a matching, and $\sum_{e \in M \cap E(G)} \wei(e) = \sum_{e \in M' \cap E(G_1)} \wei'(e) \geq k$.
It remains to verify that for every vertex $u$ of $G$, such that $u \in U$, there exists $w \in N(u)$ such that $uw \in M$.
Since $M'$ is perfect, there exists $x \in V(G')$ such that $u_1x \in M'$.
However, as $u_1$ has no neighbors outside $G_1$, $x=w_1$ for some $w \in V(G)$.
Therefore, $uw \in M$ by definition of $M$.
\end{claimproof}

\noindent\textbf{Step 2:} Let $p=\max_{e \in E(G')\wei'(e)}$, $m=p \cdot |E(G')|+1$, and $n=|V(G')|/2$ (recall that $|V(G')|$ is even).
For every $e \in E(G')$ we define $\wei''(e)=\wei'(e)+m$, and we set $k'=nm+2k$.

\begin{claim}\label{cla:matching-too}
There exists a perfect matching $M'$ in $G'$ such that $\wei'(M')\geq 2k$ if and only if $(G',\wei'',k')$ is a yes-instance of \textsc{Maximum Weight Matching}.
\end{claim}
\begin{claimproof}
Assume that if there exists a perfect matching $M'$ in $G'$ such that $\sum_{e \in M'} \wei'(e) \geq 2k$.
Since $M'$ is perfect, $|M'|=n$.
Hence,
\begin{align*}
\sum_{e \in M'} \wei''(e) &= \sum_{e \in M'} \left( m+\wei'(e) \right) \geq nm + \sum_{e \in M'} \wei'(e) \geq nm + 2k.
\end{align*}
We claim that the opposite holds: i.e., if $M'$ is a matching in $G'$ such that $\sum_{e \in M'} \wei''(e) \geq nm + 2k$, then $M'$ is perfect, and $\sum_{e \in M'} \wei'(e) \geq 2k$.
Indeed,
\begin{align}\label{eq:matching}
\sum_{e \in M'} \wei'(e) =  \sum_{e \in M'} \left( \wei''(e)-m \right) = \sum_{e \in M'} \wei''(e) - |M'| \cdot m \geq  m(n-|M'|)+2k \geq 2k,
\end{align}
and since $ \sum_{e \in M'} \wei'(e) \leq p \cdot |M'| \leq m-1$, then if $|M'| <n$:
\begin{align*}
\sum_{e \in M'} \wei''(e) =m\cdot |M'| + \sum_{e \in M'} \wei'(e) \leq m(n-1) + (m - 1) = nm -1,
\end{align*}
a contradiction with $\sum_{e \in M'} \wei''(e) \geq  nm+2k$.
Hence, $|M'| \geq n = |V(G')|/2$, i.e., $M'$ is perfect.
\end{claimproof}

By \cref{cla:matching-one} and \cref{cla:matching-too}, $(G,U,\wei,k)$ is a yes-instance of \MWM if and only if $(G',\wei'',k')$ is a yes-instance of \textsc{Maximum Weight Matching}.
Since the latter can be checked in polynomial time by the algorithm of Edmonds~\cite{edmonds_1965}, and $(G',\wei'',k')$ can be constructed in time polynomial in $|V(G)|$, we obtain the answer for $(G,U,\wei,k)$ in polynomial time.
\end{proof}
\section*{Appendix. \textsc{Pos-1-in-3-SAT} with each variable occuring a constant number of times}\label{sec:appendixSAT}
\oneinthreesat*

\begin{proof}
Consider an instance of \textsc{3-Sat} with $n$ variables and $m$ clauses.
First, by the result of Tovey~\cite{DBLP:journals/dam/Tovey84}, we create an equivalent instance of \textsc{3-Sat}, with at most $\Oh(n+m)$ variables and at most $\Oh(m)$ clauses, and such that each variable appears in at most $4$ clauses.
Then, using the reduction by Schaeffer~\cite[Lemma 3.5]{DBLP:conf/stoc/Schaefer78}, we obtain an equivalent instance $\cI$ of \textsc{1-in-3-Sat}, with $\Oh(n+m)$ variables and at most $\Oh(m)$ clauses, and such that each variable appears in at most $4$ clauses.

Finally, we create an instance $\cI'=(X,\cC)$ of \textsc{Pos-1-in-3-Sat} from $\cI$ as follows.
For every variable $x$ of $\cI$, we introduce five variables $x_0,x_1,a,b,c$.
We replace every negative occurrence of $x$ in a clause by $x_0$, and each positive occurrence by $x_1$, and we add three clauses $\{x_0,x_1,a\}, \{x_0,x_1,b\},\{a,b,c\}$. 
It is straightforward to verify that the instance constructed in this way is equivalent to $\cI$, in which each variable appears in a most $6$ clauses. Moreover, the number of variables and clauses is $\Oh(n+m)$.

Hence, an algorithm solving $\cI'$ in time $2^{o(|X| + |\cC|)}$ could be used to solve the initial instance of \textsc{3-Sat} in time $2^{o(n+m)}$, which contradicts the ETH (see e.g., \cite[Theorem 14.4]{DBLP:books/sp/CyganFKLMPPS15}).
\end{proof}


%

\end{document}